\documentclass[11pt]{article} 

\usepackage[utf8]{inputenc} 

\usepackage{geometry} 
\usepackage{graphicx} 

\usepackage{booktabs} 
\usepackage{array} 
\usepackage{paralist} 
\usepackage{verbatim} 
\usepackage{subfig} 
\usepackage{amsmath}
\usepackage{amssymb}
\usepackage{amsthm}
\usepackage{mathtools}
\usepackage{bbm}
\usepackage{inputenc}
\usepackage{mathrsfs}

\usepackage{fancyhdr} 
\usepackage{sectsty}

\usepackage[nottoc,notlof,notlot]{tocbibind} 
\usepackage[titles,subfigure]{tocloft} 

\numberwithin{equation}{section}

\theoremstyle{plain}
\newtheorem{thm}{Theorem}[section]
\newtheorem{lem}[thm]{Lemma}
\newtheorem{prop}[thm]{Proposition}
\newtheorem{cor}[thm]{Corollary}
\theoremstyle{definition}
\newtheorem{defn}[thm]{Definition}
\newtheorem{exmp}[thm]{Example}
\newtheorem{hyp}[thm]{Hypothesis}
\theoremstyle{remark}
\newtheorem{rem}[thm]{Remark}

\def\Acal{\mathcal{A}}
\def\Bcal{\mathcal{B}}

\def\Dcal{\mathcal{D}}
\def\Ecal{\mathcal{E}}
\def\Fcal{\mathcal{F}}

\def\Hcal{\mathcal{H}}
\def\Lcal{\mathcal{L}}
\def\Mcal{\mathcal{M}}
\def\Pcal{\mathcal{P}}
\def\Scal{\mathcal{S}}

\def\Ebb{\mathbb{E}}
\def\Fbb{\mathbb{F}}
\def\Nbb{\mathbb{N}}
\def\Pbb{\mathbb{P}}

\def\Rbb{\mathbb{R}}

\def\Wbb{\mathbb{W}}

\def\1bb{\mathbbm{1}}
\def\Hscr{\mathscr{H}}
\def\Kscr{\mathscr{K}}

\DeclareMathOperator*{\esssup}{ess\,sup}
\DeclareMathOperator{\sto}{sto}
\let\epsilon\varepsilon
\let\phi\varphi

\newcommand{\vertiii}[1]{{\left\vert\kern-0.25ex\left\vert\kern-0.25ex\left\vert #1 
\right\vert\kern-0.25ex\right\vert\kern-0.25ex\right\vert}}

\title{Continuous random field solutions to \\ parabolic SPDEs on p.c.f. fractals}
\author{Ben Hambly\footnote{Mathematical Institute, University of Oxford, Woodstock Road, Oxford, OX2 6GG, UK. Email: hambly@maths.ox.ac.uk.}\ \ 
and Weiye Yang\footnote{Mathematical Institute, University of Oxford, Woodstock Road, Oxford, OX2 6GG, UK. Email: weiye.yang@maths.ox.ac.uk. 
ORCiD: 0000-0003-2104-1218.}}
\date{} 

\begin{document}
\maketitle
\begin{abstract}
We consider a general class of $L^2$-valued stochastic processes that arise primarily as solutions of parabolic SPDEs on p.c.f. fractals. Using a 
Kolmogorov-type continuity theorem, conditions are found under which these processes admit versions which are function-valued and jointly 
continuous in space and time, and the associated H\"older exponents are computed. We apply this theorem to the solutions of SPDEs in the 
theories of both da Prato--Zabczyk and Walsh. We conclude by discussing a version of the parabolic Anderson model on these fractals and demonstrate 
a weak form of intermittency.
\end{abstract}

\section{Introduction}

In the study of the analytic properties of fractal sets, there is a large body of work on so called finitely ramified fractals. These are fractals which 
can be disconnected by the removal of a finite number of points and the Sierpinski gasket is an example. This gives them analytic properties 
closer to sets in one dimensional Euclidean space than to sets in higher dimensions. Our aim is to investigate a broad class of SPDEs on such finitely
ramified sets and for this we will work within the class of p.c.f. self-similar sets as introduced by Kigami, see \cite{Kigami2001}.

Let $(F,(\psi_i)_{i=1}^N)$ be a p.c.f.s.s. set which admits a regular harmonic structure in the sense of \cite{Kigami2001}. This harmonic 
structure is associated with a collection of Laplace operators $\{\Delta_b\}_b$ on $L^2(F)$ (with respect to a natural measure on $F$ to be 
specified later), where $b$ denotes a boundary condition with respect to the space $F$. Each of these Laplacians has an associated contraction 
semigroup $S^b$. We are concerned with stochastic processes $U = (U(t):t \in [0,T])$ on $L^2(F)$ of the form
\begin{equation}\label{mild}
U(t) = \int_0^t S^b_{t-s} \beta_s ds + \int_0^t S^b_{t-s} \sigma_s dW(s)
\end{equation}
for given boundary condition $b$, where $W$ is a cylindrical Wiener process on $L^2(F)$ and $\beta$ and $\sigma$ are processes satisfying 
minimal regularity conditions for the above integrals to make sense for all $t \in [0,T]$. Processes of this form arise as the mild solutions of 
parabolic stochastic partial differential equations on $L^2(F)$, for example the stochastic heat equation 
\begin{equation}\label{she}
dU(t) = \Delta_b U(t)dt + dW(t), \quad t \in [0,T]
\end{equation}
and a version of the parabolic Anderson model
\begin{equation}\label{pam}
\frac{\partial u}{\partial t}(t,x) = \Delta_b u(t,x) + u(t,x)\xi(t,x), \quad (t,x) \in [0,T] \times F,
\end{equation}
where $\xi$ is a space-time white noise on $F$. Note that \eqref{she} is to be interpreted as a stochastic evolution equation on the separable 
Hilbert space $L^2(F)$ in the sense of da Prato and Zabczyk \cite{DaPrato1992}, so its solution is of the form \eqref{mild} by definition. On 
the other hand, \eqref{pam} is a parabolic SPDE on $F$ in the sense of Walsh \cite{Walsh1986}. There is sufficient overlap between these 
two theories on the space $F$ that even the solution of \eqref{pam} can be interpreted as an $L^2(F)$-valued process of the form 
\eqref{mild}---this is explained in detail in Section \ref{applications2}. The main question we address is the following: under what conditions 
does there exist a measurable function $u: \Omega \times [0,T] \times F \to \Rbb$ which is almost surely continuous on $[0,T] \times F$ and 
such that the process $(u(t,\cdot): t \in [0,T])$ is a version of $U$ in \eqref{mild}?

A well-known result in the vein of continuity of solutions to SPDEs concerns the stochastic heat equation in one Euclidean spatial 
dimension, see for example \cite{Walsh1986}. The result is that this equation has a unique mild solution that is essentially 
$\frac{1}{2}$-H\"older continuous in space and essentially $\frac{1}{4}$-H\"older continuous in time. Here 
\textit{essentially $\gamma$-H\"older continuous} means that the given function is H\"older continuous for every exponent strictly less 
than $\gamma$. This result has since been extended in numerous directions, for example \cite{Khoshnevisan2009}, \cite{Foondun2011}, 
\cite{Khoshnevisan2014}, all of which preserve the condition that the underlying space be Euclidean. In \cite{Hambly2016} an analogous 
result is found where the ``spatial'' Euclidean space $\Rbb$ is replaced with a p.c.f.s.s. set $F$ equipped with a regular harmonic structure which
determines a natural metric, called the resistance metric on $F$. 
It is shown in that paper that the stochastic heat equation \eqref{she} on $F$ has a unique mild solution, and that this solution is 
function-valued and (has a version which is) almost surely H\"older continuous in $[0,T] \times F$. We call such a solution a 
\textit{continuous random field} solution. The aim of the present paper is to generalise this result; we seek conditions on $\beta$ and 
$\sigma$ such that the process $U$ given by \eqref{mild} has a version that is a continuous random field, and moreover we derive the 
dependence of its H\"older exponents on the properties of $\beta$, $\sigma$ and $F$. We then apply this result to derive H\"older continuity 
properties of the solutions of SPDEs on $F$ in the theories of both da Prato--Zabczyk and Walsh. In this way we have a single regularity 
result that allows us to prove properties of SPDEs in these two different theories. 

The main theorem of the present paper uses a Kolmogorov-type continuity theorem for stochastic processes indexed by $[0,T] \times F$, 
and extends the result of \cite{Hambly2016} in multiple directions; we use general multiplicative drift and diffusion terms and increase the 
number of boundary conditions that can be considered. Additionally, it can be seen that the unit interval $[0,1]$ has an interpretation as a 
p.c.f. fractal and its associated resistance metric is simply the Euclidean metric, so our theorem thus provides an alternative proof for 
existence of continuous random field versions of SPDEs on $[0,1]$, one that does not rely on Fourier transforms or on the explicit form of 
the heat kernel. A significant issue that arises is the fact that the random field version of $U$ that we construct does not, in general, have 
an explicit expression in terms of known quantities. This contrasts with the case of the stochastic heat equation, the solution of which 
possesses an explicit representation in terms of the spectrum of $\Delta_b$ (\cite[Section 4]{Hambly2016}). This means that certain nice 
properties such as joint measurability of the solution do not come cheaply. A special case that we consider is the situation where the diffusion 
coefficient $\sigma$ is a time-dependent multiplication operator on $L^2(F)$; in such a situation we find that a somewhat weaker condition 
on $\sigma$ is sufficient to obtain the same result. This is the case that corresponds to Walsh-type SPDEs.

Our results cover the version of the parabolic Anderson model over a compact space with time 
dependent potential as given in \eqref{pam}. We will investigate this further  by considering the question of intermittency---does the 
solution exhibit tall peaks over infinitely many different scales? This is shown by the exponential growth of the moments in that the function 
$\lambda(p)$, defined by
\[ \lambda(p) = \lim_{t\to\infty} \frac{1}{t} \log \Ebb |u(t,x)|^p, \]
should have the property that $p\to\lambda(p)/p$ is strictly increasing on $[2,\infty)$. For a weak form of intermittency it is enough to establish that 
$\lambda(2)>0$, \cite{Khoshnevisan2014}, and this is what will do for our model in Theorem~\ref{thm:2ndlower}. We conjecture that, for 
p.c.f. fractals $F$, we have $\lambda(p) = \Theta(p^{2+d_H})$ for large $p$, where $d_H$ is the Hausdorff dimension of $F$ with respect to 
the resistance metric and can take any value in $[1,\infty)$.

The present paper is organised in the following way: In Section \ref{setup} we set up the problem precisely and state the main hypotheses 
and theorem. In Section \ref{randomfield} we construct from the $L^2(F)$-valued process $U$ a candidate collection of random variables 
$(\tilde{u}(t,x):(t,x) \in [0,T] \times F)$. Then in Section \ref{cty} we use the previously mentioned continuity theorem to construct a H\"older 
continuous version $u$ of $\tilde{u}$, which we identify to be a continuous random field version of $U$ and thus we prove the main theorem. 
In Sections \ref{applications1} and \ref{applications2} we give applications of our main theorem to two classes of stochastic partial 
differential equations; these are defined with respect to the SPDE theories developed by da Prato--Zabczyk~\cite{DaPrato1992} and 
Walsh~\cite{Walsh1986} respectively. Finally, in Section~\ref{sec:intermitt}, in the Walsh case, we also prove a number of upper and lower bounds on the moments of global 
solutions under various hypotheses. In particular we establish weak intermittency for the parabolic Anderson model on these fractals.

\section{Set-up and statement of main result}\label{setup}

\subsection{The fractal}

Our set-up is similar to that of \cite{Hambly2016}. Let $N \geq 2$ be an integer, and let $(F,(\psi_i)_{i=1}^N)$ be a connected compact 
post-critically finite self-similar (p.c.f.s.s.) set such that each $\psi_i$ is a strict contraction on $F$ (see \cite{Kigami2001} for details). 
Let $I = \{ 1,\ldots, N \}$, and for $n \geq 0$ let $\Wbb_n = I^n$. Let $\Wbb_* = \bigcup_{n = 0}^\infty \Wbb_n$ and let $\Wbb = I^\Nbb$. 
Elements of $\Wbb_*$ are called \textit{words} and the element of the singleton $\Wbb_0$ is known as the \textit{empty word}. For $w = 
w_1\ldots w_n \in \Wbb_n$ let $\psi_w = \psi_{w_1} \circ \cdots \circ \psi_{w_n}$ and let $F_w = \psi_w(F)$. There is a canonical continuous 
surjection $\pi: \Wbb \to F$ given by defining $\pi(w)$ to be the element of the singleton $ \bigcap_{m=1}^\infty F_{w_1 \ldots w_m}$ 
(see~\cite[Lemma 5.10]{Barlow1998}). Let $P$ be the post-critical set of $(F,(\psi_i)_{i=1}^N)$ and, by assumption this is a finite set. Now 
set $F^0 = \pi(P)$. In view of \cite[Proposition 1.3.5(2)]{Kigami2001}, $F^0$ can be said to be the \textit{boundary} of $F$. The standard 
example is the Sierpinski gasket, but this framework captures the idea of finitely ramified fractals, that is self-similar sets that can be 
disconnected by the removal of a finite number of points.

Let $(A,\textbf{r})$ be a regular irreducible harmonic structure on $F$ (\cite[Section 3.1]{Kigami2001}). Since it is regular we have that 
$\textbf{r} = (r_1,\ldots,r_N)$ with $r_i \in (0,1)$ for each $i \in I$. Let $r_{\max} = \max_{i \in I}r_i$ and $r_{\min} = \min_{i \in I}r_i$ and 
for $w \in \Wbb_m$ let $r_w = r_{w_1}\ldots r_{w_m}$. Let $d_H > 0$ be the unique number such that
\begin{equation*}
\sum_{i \in I} r_i^{d_H} = 1.
\end{equation*}
Let $\mu$ be the self-similar measure on $F$ with weights given by $(r_i^{d_H})_{i \in I}$. Set $\Hcal = L^2(F,\mu)$ with inner product 
$\langle\cdot,\cdot\rangle_\mu$. Then the harmonic structure $(A,\textbf{r})$ is associated with a local regular Dirichlet form $(\Ecal,\Dcal)$ 
on $\Hcal$. This Dirichlet form provides a resistance metric $R$ on $F$ which is compatible with its original topology 
(\cite[Thoerem 3.3.4]{Kigami2001}), and it can be seen that the constant $d_H$ is in fact the Hausdorff dimension of $(F,R)$ 
(\cite[Theorem 4.2.1]{Kigami2001}).

Let $2^{F^0}$ be the power set of $F^0$. For $b \in 2^{F^0}$ define $\Dcal_b = \{ f \in \Dcal: f|_{F^0 \setminus b} = 0 \}$. Then 
$(\Ecal,\Dcal_b)$ is a local regular Dirichlet form on $L^2(F \setminus (F^0 \setminus b), \mu)$ which is naturally associated with a diffusion 
$X^b = (X^b_t)_{t \geq 0}$, which has semigroup $S^b = (S^b_t)_{t \geq 0}$ and generator $\Delta_b$, see \cite{Fukushima2011}. The 
latter is known as the \textit{Laplacian}, as $(-\Delta_b)$ is the self-adjoint operator associated with the closed form $(\Ecal,\Dcal_b)$ (see 
\cite[Sections 1.3 and 1.4]{Fukushima2011}). The value of $b \in 2^{F^0}$ represents the boundary condition that we impose on the 
Laplacian---the operator $\Delta_b$ is the Laplacian with Neumann boundary conditions at elements of $b \subseteq F^0$ and Dirichlet 
boundary conditions at elements of $F^0 \setminus b$. The case $b = F^0$ therefore denotes Neumann boundary conditions and the case 
$b = \emptyset$ denotes Dirichlet boundary conditions. We may occasionally use the notation $b = N$ instead of $b = F^0$ and likewise 
we may use $b = D$ instead of $b = \emptyset$.

\subsection{Preliminaries}

We start with a couple of very useful results on the operator $\Delta_b$. Let $\Lcal(\Hcal)$ be the Banach space of bounded linear 
operators on $\Hcal$ with operator norm $\Vert \cdot \Vert$. If $\Hcal_1, \Hcal_2$ are Hilbert spaces, let $\Lcal_2(\Hcal_1,\Hcal_2)$ be the 
Hilbert space of Hilbert-Schmidt operators from $\Hcal_1$ to $\Hcal_2$ with inner product $\langle \cdot,\cdot\rangle_{\Lcal_2(\Hcal_1,\Hcal_2)}$. 
If $A$ is a linear operator on $\Hcal$ then we denote the domain of $A$ by $\Dcal(A)$.
\begin{defn}
For $\lambda > 0$ let $\Dcal^\lambda$ be the space $\Dcal$ equipped with the inner product
\begin{equation*}
\langle \cdot,\cdot \rangle_\lambda := \Ecal(\cdot,\cdot) + \lambda \langle \cdot,\cdot \rangle_\mu.
\end{equation*}
Since $(\Ecal,\Dcal)$ is closed, $\Dcal^\lambda$ is a Hilbert space.
\end{defn}

\begin{rem}\label{domcl}
The space $\Dcal$ contains only $\frac{1}{2}$-H\"{o}lder continuous functions since by the definition of the resistance metric we have that
\begin{equation}
|f(x) - f(y)|^2 \leq R(x,y) \Ecal(f,f)
\end{equation}
for all $f \in \Dcal$ and all $x,y \in F$. We deduce that for each $b \in 2^{F^0}$, $\Dcal_b$ is a subspace of $\Dcal^\lambda$ with finite 
codimension $|F^0 \setminus b|$. Indeed, $\Dcal_b$ is the intersection of the kernels of the set of evaluation functionals $\{ f \mapsto f(x): x 
\in F^0 \setminus b \}$, which are linear and continuous and linearly independent on $\Dcal^\lambda$. It follows that $\Dcal_b$ is closed in 
$\Dcal^\lambda$.
\end{rem}

\begin{defn}
The unique real $d_H > 0$ such that
\begin{equation*}
\sum_{i \in I} r_i^{d_H} = 1
\end{equation*}
is the \textit{Hausdorff dimension} of $(F,R)$, see \cite[Theorem 1.5.7]{Kigami2001}.

The \textit{spectral dimension} of $(F,R)$ is given by
\begin{equation*}
d_s = \frac{2d_H}{d_H + 1},
\end{equation*}
see \cite[Theorem 4.1.5 and Theorem 4.2.1]{Kigami2001}.
\end{defn}

\begin{rem}
It is possible to show that we must have $d_H \in [1,\infty)$ and thus $d_s \in [1,2)$, see \cite[Remark 2.6(2)]{Hambly2016}.
\end{rem}

\begin{prop}[Spectral theory]\label{spectra}
For $b \in 2^{F^0}$ the following statements hold: 

There exists a complete orthonormal basis $(\phi^b_k)_{k=1}^\infty$ of $\Hcal$ consisting of eigenfunctions of the operator $-\Delta_b$. 
The corresponding eigenvalues $(\lambda^b_k)_{k=1}^\infty$ are non-negative and $\lim_{k \to \infty}\lambda^b_k = \infty$. We assume 
that they are given in ascending order:
\begin{equation*}
0 \leq \lambda^b_1 \leq \lambda^b_2 \leq \cdots.
\end{equation*}
There exist constants $c_0,c_0' > 0$ such that if $k \geq 2$ then 
\begin{equation*}
c_0k^\frac{2}{d_s} \leq \lambda^b_k \leq c_0'k^\frac{2}{d_s}.
\end{equation*}
\end{prop}

\begin{proof}
With \cite[Lemma 5.1.3]{Kigami2001} in mind, this is immediate from \cite[Proposition A.2.11]{Kigami2001} and subsequent discussion.
\end{proof}

\begin{rem}
We have $k \geq 2$ in the above proposition because of the possibility that $\lambda^b_1 = 0$. This occurs if and only if the non-zero constant functions are elements of $\Dcal_b$, if and only if $b = N$. If this is the case, it follows that $\phi^b_1 \equiv 1$ and the eigenvalue $0$ has multiplicity $1$, so $\lambda^b_2 > 0$. See \cite[Remark 2.9]{Hambly2016}. In the case $b \neq N$ we assume that the constants $c_0,c_0'$ in the proposition above are chosen such that the result in fact holds for $k \geq 1$.
\end{rem}

\subsection{The class of SPDEs}

Let $(\Omega, \Fcal, \Fbb, \Pbb)$ be a filtered probability space that supports a cylindrical Wiener process $W = (W(t))_{t \geq 0}$ on $\Hcal$ 
(note that we always suppress the dependence of stochastic processes on $\omega \in \Omega$).  Let $\Fbb = (\Fcal_t)_{t \geq 0}$ be a filtration 
satisfying the usual conditions, and let $\Pcal$ be the associated predictable $\sigma$-algebra. Let $T > 0$ be fixed. We are concerned with 
stochastic partial differential equations on $\Hcal$ of the form
\begin{equation}\label{SPDE}
\begin{split}
dU(t) &= \Delta_b U(t)dt + \beta_tdt + \sigma_tdW(t)\quad t \in [0,T],\\
U(0) &= 0,
\end{split}
\end{equation}
where $\beta: \Omega \times [0,T] \to \Hcal$ and $\sigma: \Omega \times [0,T] \to \Lcal(\Hcal)$ are predictable processes, and $b \in 2^{F^0}$. 
Our question is not one of existence or uniqueness of $U$---we simply assume that a predictable $\Hcal$-valued process $U = (U(t))_{t \in [0,T]}$ 
exists and satisfies \eqref{SPDE}. In particular, we assume that $U$ satisfies \eqref{SPDE} in the \textit{mild} sense: for each $t \in [0,T]$,
\begin{equation}\label{SPDEmild}
U(t) = \int_0^t S^b_{t-s} \beta_s ds + \int_0^t S^b_{t-s} \sigma_s dW(s)
\end{equation}
almost surely. This means that a priori we require $\beta$ and $\sigma$ to be at least regular enough such that the two integrals on the right-hand 
side of the above are well-defined for all $t \in [0,T]$. To be precise, we need the following to hold almost surely for each $t \in [0,T]$:
\begin{equation}\label{aprioricond}
\begin{split}
\int_0^t \Vert S^b_{t-s} \beta_s \Vert_\mu ds &< \infty,\\
\int_0^t \Vert S^b_{t-s} \sigma_s \Vert_{\Lcal_2(\Hcal,\Hcal)}^2 ds &< \infty.
\end{split}
\end{equation}
We will shortly make hypotheses that automatically guarantee \eqref{aprioricond}, so these do not need to be checked explicitly.
\begin{rem}
While \eqref{SPDE} may seem to be an uninteresting equation, in that there is no explicit dependence of $\beta$ and $\sigma$ on $U$, it actually 
encompasses the solutions to a large class of SPDEs. For example, if we take $Y$ to be a mild solution of the SPDE
\begin{equation*}
\begin{split}
dY(t) &= \Delta_b Y(t)dt + f(t,Y(t))dt + g(t,Y(t))dW(t)\quad t \in [0,T],\\
Y(0) &= Y_0 \in \Hcal
\end{split}
\end{equation*}
where the functions $f: \Omega \times [0,T] \times \Hcal \to \Hcal$ and $g: \Omega \times [0,T] \times \Hcal \to \Lcal(\Hcal)$ satisfy appropriate 
measurability conditions, then setting $\beta_t = f(t,Y(t))$ and $\sigma_t = g(t,Y(t))$ we see that the process $U$ given by $U(t) = Y(t) - S^b_tY_0$ 
satisfies \eqref{SPDEmild}.
\end{rem}
\begin{defn}
A (square-integrable) \textit{random field} (on $[0,T] \times F$) is a function $u: \Omega \times [0,T] \times F \to \Rbb$ such that $u(t,x): \Omega 
\to \Rbb$ is a random variable for each $(t,x) \in [0,T] \times F$, and $(u(t,\cdot))_{t \in [0,T]}$ is an $\Hcal$-valued stochastic process. It is a 
\textit{continuous random field} if $u$ is, in addition, almost surely jointly continuous in $[0,T] \times F$.
\end{defn}
\begin{rem}
Any continuous random field must be jointly measurable in $\Omega \times [0,T] \times F$ by \cite[Lemma 4.51]{Aliprantis2006} and the fact 
that $[0,T] \times F$ is compact.
\end{rem}
\begin{defn}\label{essentially}
Let $(M_1,d_1)$ and $(M_2,d_2)$ be metric spaces and let $\delta \in (0,1]$. A function $f: M_1 \to M_2$ is \textit{essentially $\delta$-H\"{o}lder 
continuous} if for each $\gamma \in (0,\delta)$ there exists $C_\gamma > 0$ such that
\begin{equation*}
d_2(f(x),f(y)) \leq C_\gamma d_1(x,y)^\gamma
\end{equation*}
for all $x,y \in M_1$.
\end{defn}
Define a metric $R_\infty$ on $\Rbb \times F$ by $R_\infty((s,x),(t,y)) = |s-t| \vee R(x,y)$. We now give two separate sets of hypotheses for the 
behaviour of the processes $\beta$ and $\sigma$ and state our main theorem.
\begin{hyp}\label{hyp1}
There exist $p > (d_H + 1)^2$ and $K > 0$ such that
\begin{equation*}
\begin{split}
\Ebb\left[ \left( \int_0^T \Vert\beta_s\Vert_\mu^2 ds \right)^p \right] &\leq K,\\
\sup_{s \in [0,T]} \Ebb\left[ \Vert \sigma_s \Vert^{2p} \right] &\leq K.
\end{split}
\end{equation*}
\end{hyp}
\begin{defn}
For a measurable function $f: F \to \Rbb$, let $\Mcal_f$ be the multiplication operator on $\Hcal$ associated with $f$. That is, $\Mcal_fh$ is the 
pointwise multiplication $f \cdot h := (x \mapsto f(x)h(x))$ for all $h \in \Hcal$ such that $f \cdot h \in \Hcal$.
\end{defn}
Note that for $f: F \to \Rbb$ measurable, $\Mcal_f \in \Lcal(\Hcal)$ if and only if
\begin{equation*}
\esssup_{x \in F} |f(x)| := \inf\{ c > 0: \mu(\{ x \in F: |f(x)| > c \}) = 0 \} < \infty.
\end{equation*}
Additionally if the above inequality holds then
\begin{equation*}
\Vert \Mcal_f \Vert = \esssup_{x \in F} |f(x)|,
\end{equation*}
see \cite[Theorem II.1.5]{Conway1990}. Note also that $\esssup_{x \in F} |f(x)| \leq \sup_{x \in F} |f(x)|$.

Our second set of hypotheses is similar to the first one, but we restrict the space of operators in which the process $\sigma$ can take values 
to the set of multiplication operators. This allows us to weaken the integrability condition on $\sigma$ given in Hypothesis \ref{hyp1}.
\begin{hyp}\label{hyp2}
There exists a jointly measurable function $\tilde{\sigma}: \Omega \times [0,T] \times F \to \Rbb$ such that for each $t \in [0,T]$,
\begin{equation*}
\sigma_t = \Mcal_{\tilde{\sigma}(t,\cdot)}
\end{equation*}
almost surely. There exist $p > (d_H + 1)^2$ and $K > 0$ such that
\begin{equation*}
\begin{split}
\Ebb\left[ \left( \int_0^T \Vert\beta_s\Vert_\mu^2 ds \right)^p \right] &\leq K,\\
\sup_{(s,x) \in [0,T] \times F} \Ebb\left[ | \tilde{\sigma}(s,x) |^{2p} \right] &\leq K.
\end{split}
\end{equation*}
\end{hyp}
The following result shows that we need not check the conditions \eqref{aprioricond} if we have assumed either of our above hypotheses:
\begin{prop}
Assume either Hypothesis \ref{hyp1} or Hypothesis \ref{hyp2}. Then the conditions \eqref{aprioricond} are satisfied.
\end{prop}
\begin{proof}
Fix $t \in [0,T]$. Assume Hypothesis \ref{hyp1}. We have
\begin{equation*}
\int_0^t \Vert S^b_{t-s} \beta_s \Vert_\mu ds \leq \int_0^t \Vert  \beta_s \Vert_\mu ds < \infty
\end{equation*}
almost surely, and
\begin{equation*}
\begin{split}
\Ebb \left[ \int_0^t \Vert S^b_{t-s} \sigma_s \Vert_{\Lcal_2(\Hcal,\Hcal)}^2 ds \right] & \leq \int_0^t \Ebb \left[ \Vert \sigma_s \Vert^2 \right] \Vert S^b_{t-s} \Vert_{\Lcal_2(\Hcal,\Hcal)}^2 ds\\
&= \sup_{s \in [0,T]} \Ebb \left[ \Vert \sigma_s \Vert^2 \right] \int_0^t \sum_{k=1}^\infty \Vert S^b_s \phi^b_k \Vert_\mu^2 ds\\
&= \sup_{s \in [0,T]} \Ebb \left[ \Vert \sigma_s \Vert^2 \right] \sum_{k=1}^\infty \int_0^t e^{-2\lambda^b_k s} ds\\
&< \infty
\end{split}
\end{equation*}
where the last inequality follows by Proposition \ref{spectra} and the fact that $d_s < 2$.

Now we instead assume Hypothesis \ref{hyp2}. The condition on $\beta$ is the same as in Hypothesis \ref{hyp1}. Let $\sigma^*_s$ be the 
adjoint of $\sigma_s$ for each $s \in [0,T]$. We note that (bounded) multiplication operators are self-adjoint (\cite[Example VI.1.5]{Conway1990}), 
so we have that
\begin{equation*}
\begin{split}
\Ebb \left[ \int_0^t \Vert S^b_{t-s} \sigma_s \Vert_{\Lcal_2(\Hcal,\Hcal)}^2 ds \right] &= \int_0^t \Ebb \left[ \Vert \sigma^*_s S^b_{t-s} \Vert_{\Lcal_2(\Hcal,\Hcal)}^2 \right] ds\\
&= \int_0^t \Ebb \left[ \sum_{k=1}^\infty \Vert \Mcal_{\tilde{\sigma}(s,\cdot)} S^b_{t-s} \phi^b_k \Vert_\mu^2 \right] ds\\
&= \sum_{k=1}^\infty \int_0^t e^{-2\lambda^b_k (t-s)} \Ebb \left[ \Vert \Mcal_{\tilde{\sigma}(s,\cdot)} \phi^b_k \Vert_\mu^2 \right] ds\\
&= \sum_{k=1}^\infty \int_0^t e^{-2\lambda^b_k (t-s)} \int_F \Ebb \left[ \tilde{\sigma}(s,x)^2 \right] \phi^b_k(x)^2 \mu(dx) ds\\
&\leq \sup_{(s,x) \in [0,T] \times F}\Ebb \left[ \tilde{\sigma}(s,x)^2 \right] \sum_{k=1}^\infty \int_0^t e^{-2\lambda^b_k s} ds\\
&< \infty
\end{split}
\end{equation*}
where the last inequality follows by Proposition \ref{spectra} and the fact that $d_s < 2$.
\end{proof}
The following is the main result of the present paper:
\begin{thm}\label{mainthm}
Assume either Hypothesis \ref{hyp1} or Hypothesis \ref{hyp2}. Then the process $U$ given by \eqref{SPDEmild} has a version which is a 
continuous random field $u: \Omega \times [0,T] \times F \to \Rbb$. In particular $u$ satisfies the following:
\begin{enumerate}
\item $u$ is almost surely essentially $\frac{1}{2}\left((d_H + 1)^{-1} - \frac{d_H + 1}{p} \right)$-H\"{o}lder continuous in $[0,T] \times F$ with
 respect to $R_\infty$,
\item For each $t \in [0,T]$, $u(t,\cdot)$ is almost surely essentially $\frac{1}{2}\left(1 - \frac{d_H}{p} \right)$-H\"{o}lder continuous in $F$ with respect to $R$,
\item For each $x \in F$, $u(\cdot,x)$ is almost surely essentially $\frac{1}{2}\left((d_H + 1)^{-1} - \frac{1}{p} \right)$-H\"{o}lder continuous in $[0,T]$.
\end{enumerate}
\end{thm}
\begin{exmp}
We compare this with the H\"{o}lder continuity results proven in \cite[Theorem 5.6]{Hambly2016} for the stochastic heat equation on a p.c.f. fractal. 
In this case $\beta$ and $\sigma$ are both deterministic and constant, so in Hypothesis \ref{hyp1} we are able to take $p \to \infty$ and there is an 
additional smoothing parameter $\alpha$. In the case $\alpha = 0$, \cite[Theorem 5.6]{Hambly2016} and Theorem \ref{mainthm} then give the 
same H\"{o}lder exponents for the random field solution to the stochastic heat equation (since $(d_H + 1)^{-1} = (1 - \frac{d_s}{2})$). However this 
is not the case for general $\alpha$, and this is because in the present paper we do not take into account any potential smoothing properties of the 
process $\sigma$.
\end{exmp}

\section{Construction of random field}\label{randomfield}

\subsection{Resolvent density}

We develop and expand some theory from \cite[Section 4]{Hambly2016}:
\begin{defn}
For $\lambda > 0$ and $b \in 2^{F^0}$ let $\rho^b_\lambda: F \times F \to \Rbb$ be the \textit{resolvent density} associated with $\Delta_b$. 
By \cite[Theorem 7.20]{Barlow1998}, $\rho^N_\lambda$ exists and satisfies the following:
\begin{enumerate}
\item (Reproducing kernel property.) For $x \in F$, $\rho^N_\lambda(x,\cdot)$ is the unique element of $\Dcal_N = \Dcal$ such that
\begin{equation*}
\langle \rho^N_\lambda(x,\cdot), f \rangle_\lambda = f(x)
\end{equation*}
for all $f \in \Dcal_N$.
\item (Resolvent kernel property.) For all continuous $f \in \Hcal$ and all $x \in F$,
\begin{equation*}
\int_0^\infty e^{-\lambda t} S^N_tf(x) dt = \int_F \rho^N_\lambda(x,y)f(y)\mu(dy).
\end{equation*}
By a density argument it follows that for all $f \in \Hcal$,
\begin{equation*}
\int_0^\infty e^{-\lambda t} S^N_tf dt = \int_F \rho^N_\lambda(\cdot,y)f(y)\mu(dy).
\end{equation*}
\item $\rho^N_\lambda$ is symmetric and bounded. We define (for now) $c_\rho(\lambda) > 0$ such that
\begin{equation*}
c_\rho(\lambda) \geq \sup_{x,y \in F}|\rho^N_\lambda(x,y)|.
\end{equation*}
\item (H\"{o}lder continuity.) For this same constant $c_\rho(\lambda)$ we have that for all $x,y,y' \in F$,
\begin{equation*}
|\rho^N_\lambda(x,y) - \rho^N_\lambda(x,y')|^2 \leq c_\rho(\lambda) R(y,y').
\end{equation*}
Using symmetry this H\"{o}lder continuity result holds in the first argument as well.
\end{enumerate}
By an identical argument to \cite[Theorem 7.20]{Barlow1998}, $\rho^b_\lambda$ exists for each $b \in 2^{F^0}$ and satisfies the analogous results with $(\Ecal,\Dcal_b)$ and the semigroup $S^b$. By the reproducing kernel property it follows that for every $x \in F$, $\rho^b_\lambda(x,\cdot)$ must be the $\Dcal^\lambda$-orthogonal projection of $\rho^N_\lambda(x,\cdot)$ onto $\Dcal_b$. We now choose $c_\rho(\lambda)$ large enough that it does not depend on the value of $b \in 2^{F^0}$ for (3) and (4).
\end{defn}
\begin{prop}[Lipschitz resolvent]\label{resolvLip}
If $\lambda > 0$, $b \in 2^{F^0}$ and $x,x',y \in F$ then
\begin{equation*}
\left\vert \rho^b_\lambda(x',y) - \rho^b_\lambda(x,y) \right\vert \leq 2 R(x,x').
\end{equation*}
\end{prop}
\begin{proof}
The cases $b \in \{N,D\}$ are treated in \cite[Proposition 4.5]{Hambly2016}. All other cases can be proven in the same way as the $b = D$ case, using the Green function with boundary $F^0 \setminus b$.
\end{proof}
We will henceforth be using the resolvent density exclusively in the case $\lambda = 1$, so let $c_\rho = c_\rho(1)$.

\subsection{A priori estimates}

The regularity of the resolvent density leads us to the first a priori estimates on the behaviour of \eqref{SPDEmild}, which make clear the significance of Hypotheses \ref{hyp1} and \ref{hyp2}:
\begin{prop}\label{apriori}
Let $t \in [0,T]$ and $q \geq 1$.
\begin{enumerate}
\item For all $b \in 2^{F^0}$ and $h \in \Hcal$ we have that
\begin{equation*}
\begin{split}
\Ebb &\left[ \left| \left\langle \int_0^t S^b_{t-s} \beta_s ds, h \right\rangle_\mu \right|^{2q} \right] \\
&\leq \frac{e^{2qt}}{2^q} \Ebb\left[ \left( \int_0^t \Vert\beta_s\Vert_\mu^2 ds \right)^q \right] \left( \int_F \int_F \rho^b_1(x,y)h(x)h(y)\mu(dx)\mu(dy) \right)^q.
\end{split}
\end{equation*}
\item Moreover there exists a constant $C_q > 0$ depending only on $q$ such that if $\sup_{s \in [0,T]} \Ebb\left[ \Vert \sigma_s \Vert^{2q} \right] < \infty$ then for all $b \in 2^{F^0}$ and $h \in \Hcal$ we have that
\begin{equation*}
\begin{split}
\Ebb &\left[ \left| \left\langle \int_0^t S^b_{t-s} \sigma_s dW(s), h \right\rangle_\mu \right|^{2q} \right]\\
&\leq C_q e^{2qt} \sup_{s \in [0,t]} \Ebb\left[ \Vert \sigma_s \Vert^{2q} \right] \left( \int_F \int_F \rho^b_1(x,y)h(x)h(y)\mu(dx)\mu(dy) \right)^q.
\end{split}
\end{equation*}
\item Finally, for the same constant $C_q > 0$ the following holds: suppose there exists a jointly measurable function $\tilde{\sigma}: \Omega \times [0,T] \times F \to \Rbb$ such that for each $t \in [0,T]$,
\begin{equation*}
\sigma_t = \Mcal_{\tilde{\sigma}(t,\cdot)}
\end{equation*}
almost surely. Then if $\sup_{(s,x) \in [0,T] \times F} \Ebb\left[ \left| \tilde{\sigma}(s,x) \right|^{2q} \right] < \infty$ then for all $b \in 2^{F^0}$ and $h \in \Hcal$ we have that
\begin{equation*}
\begin{split}
\Ebb &\left[ \left| \left\langle \int_0^t S^b_{t-s} \sigma_s dW(s), h \right\rangle_\mu \right|^{2q} \right]\\
&\leq C_q e^{2qt} \sup_{(s,x) \in [0,t] \times F} \Ebb\left[ \left| \tilde{\sigma}(s,x) \right|^{2q} \right] \left( \int_F \int_F \rho^b_1(x,y)h(x)h(y)\mu(dx)\mu(dy) \right)^q.
\end{split}
\end{equation*}
\end{enumerate}
\end{prop}

\begin{proof}
First of all we adapt the proof of \cite[Lemma 4.6]{Hambly2016}; we see that if $b \in 2^{F^0}$, $t \geq 0$ and $h \in \Hcal$ then
\begin{eqnarray}
\int_0^t \Vert S^b_s h \Vert_\mu^2 ds &\leq & e^{2t}\int_0^\infty e^{-2s} \Vert S^b_s h \Vert_\mu^2 ds \nonumber \\
&=& e^{2t}\left\langle \int_0^\infty e^{-2s} S^b_{2s}h ds , h \right\rangle_\mu \nonumber \\
&=& \frac{e^{2t}}{2} \left\langle \int_F \rho^b_1(\cdot,y) h(y) \mu(dy) , h \right\rangle_\mu \nonumber \\
&= &\frac{e^{2t}}{2} \int_F \int_F \rho^b_1(x,y)h(x)h(y)\mu(dx)\mu(dy). \label{resolvint}
\end{eqnarray}
As a side-effect, this shows that the double integral on the right-hand side is non-negative.
\begin{enumerate}
\item Using Cauchy-Schwarz and H\"{o}lder's inequalities we find that
\begin{equation*}
\begin{split}
\Ebb\left[ \left| \left\langle \int_0^t S^b_{t-s} \beta_s ds, h \right\rangle_\mu \right|^{2q} \right] &= \Ebb\left[ \left| \int_0^t \left\langle \beta_s, S^b_{t-s} h \right\rangle_\mu ds \right|^{2q} \right]\\
&\leq \Ebb\left[ \left( \int_0^t \Vert\beta_s\Vert_\mu \Vert S^b_{t-s} h \Vert_\mu ds \right)^{2q} \right]\\
&\leq \Ebb\left[ \left( \int_0^t \Vert\beta_s\Vert_\mu^2 ds \right)^q \right] \left( \int_0^t \Vert S^b_s h \Vert_\mu^2 ds \right)^q \\
\end{split}
\end{equation*}
and then \eqref{resolvint} implies the first required estimate.
\item Let $h^*$ be the linear functional on $\Hcal$ given by $f \mapsto \langle f,h \rangle_\mu$. Then
\begin{equation*}
\Ebb\left[ \left| \left\langle \int_0^t S^b_{t-s} \sigma_s dW(s), h \right\rangle_\mu \right|^{2q} \right] = \Ebb\left[ \left| \int_0^t h^* S^b_{t-s} \sigma_s dW(s) \right|^{2q} \right].
\end{equation*}
Let $\sigma^*_s$ be the adjoint of $\sigma_s$ for each $s \in [0,T]$. We see that 
\begin{equation*}
\begin{split}
\Ebb \left[ \int_0^t \Vert h^* S^b_{t-s} \sigma_s \Vert_{\Lcal_2(\Hcal,\Rbb)}^2 ds \right] &= \Ebb \left[ \int_0^t \Vert \sigma^*_s S^b_{t-s }h \Vert_\mu^2 ds \right]\\ 
&\leq \Vert h \Vert_\mu^2 \int_0^t \Ebb \left[ \Vert \sigma_s\Vert^2 \right] ds\\
&< \infty
\end{split}
\end{equation*}
which implies that $t' \mapsto \int_0^{t'} h^* S^b_{t-s} \sigma_s dW(s)$ is a real-valued square-integrable martingale for $t' \in [0,t]$. Then by the Burkholder-Davis-Gundy inequality there exists a constant $C_q' > 0$ depending only on $q$ such that
\begin{equation*}
\begin{split}
\Ebb\left[ \left| \int_0^t h^* S^b_{t-s} \sigma_s dW(s) \right|^{2q} \right]^\frac{1}{q} &\leq C_q' \Ebb\left[ \left(\int_0^t \Vert h^* S^b_{t-s} \sigma_s \Vert_{\Lcal_2(\Hcal,\Rbb)}^2 ds \right)^q \right]^\frac{1}{q}\\
&= C_q' \Ebb\left[ \left(\int_0^t \Vert \sigma^*_s S^b_{t-s}h \Vert_\mu^2 ds \right)^q \right]^\frac{1}{q}\\
&\leq C_q' \Ebb\left[ \left(\int_0^t \Vert \sigma_s \Vert^2 \Vert S^b_{t-s}h \Vert_\mu^2 ds \right)^q \right]^\frac{1}{q}\\
&\leq C_q' \int_0^t \Ebb\left[ \Vert \sigma_s \Vert^{2q} \right]^\frac{1}{q} \Vert S^b_{t-s}h \Vert_\mu^2 ds \\
&\leq C_q' \sup_{s \in [0,t]} \Ebb\left[ \Vert \sigma_s \Vert^{2q} \right]^\frac{1}{q} \int_0^t \Vert S^b_s h \Vert_\mu^2 ds ,\\
\end{split}
\end{equation*}
where we have used the Minkowski integral inequality. Take powers of $q$ on both sides, then \eqref{resolvint} implies the second estimate.
\item Now we assume that there exists a jointly measurable function $\tilde{\sigma}: \Omega \times [0,T] \times F \to \Rbb$ such that for each $t \in [0,T]$,
\begin{equation*}
\sigma_t = \Mcal_{\tilde{\sigma}(t,\cdot)}
\end{equation*}
almost surely. We note that (bounded) multiplication operators are self-adjoint (\cite[Example VI.1.5]{Conway1990}), so we have that
\begin{equation*}
\begin{split}
\Ebb \left[\int_0^t \Vert h^* S^b_{t-s} \sigma_s \Vert_{\Lcal_2(\Hcal,\Rbb)}^2 ds \right] &= \int_0^t \Ebb \left[ \left\Vert \sigma^*_s S^b_{t-s}h \right\Vert_\mu^2 \right] ds\\
&= \int_0^t \Ebb \left[ \left\Vert \Mcal_{\tilde{\sigma}(s,\cdot)} S^b_{t-s}h \right\Vert_\mu^2 \right] ds\\
&= \int_0^t \Ebb \left[ \int_F \tilde{\sigma}(s,x)^2 (S^b_{t-s}h)(x)^2 \mu(dx) \right] ds\\
&\leq \sup_{(s,x) \in [0,T] \times F} \Ebb \left[ \tilde{\sigma}(s,x)^2 \right] \int_0^t \int_F (S^b_s h)(x)^2 \mu(dx) ds\\
&\leq t \Vert h \Vert_\mu^2 \sup_{(s,x) \in [0,T] \times F} \Ebb \left[ \tilde{\sigma}(s,x)^2 \right]\\
&< \infty,
\end{split}
\end{equation*}
which as before implies that $t' \mapsto \int_0^{t'} h^* S^b_{t-s} \sigma_s dW(s)$ is a real-valued square-integrable martingale for $t' \in [0,t]$. As in (2) we then have that
\begin{equation*}
\Ebb\left[ \left| \int_0^t h^* S^b_{t-s} \sigma_s dW(s) \right|^{2q} \right]^\frac{1}{q} \leq C_q' \Ebb\left[ \left(\int_0^t \Vert \sigma^*_s S^b_{t-s}h \Vert_\mu^2 ds \right)^q \right]^\frac{1}{q}
\end{equation*}
with the same constant $C_q'$. We apply the Minkowski integral inequality twice:
\begin{equation*}
\begin{split}
\Ebb\left[ \left(\int_0^t \Vert \sigma^*_s S^b_{t-s}h \Vert_\mu^2 ds \right)^q \right]^\frac{1}{q} &\leq \int_0^t \Ebb\left[ \left\Vert \sigma^*_s S^b_{t-s}h \right\Vert_\mu^{2q} \right]^\frac{1}{q} ds\\
&= \int_0^t \Ebb\left[ \left\Vert \Mcal_{\tilde{\sigma}(s,\cdot)} S^b_{t-s}h \right\Vert_\mu^{2q} \right]^\frac{1}{q} ds\\
&= \int_0^t \Ebb\left[ \left( \int_F \tilde{\sigma}(s,x)^2 (S^b_{t-s}h)(x)^2 \mu(dx) \right)^q \right]^\frac{1}{q} ds\\
&\leq \int_0^t \int_F \Ebb\left[ \left| \tilde{\sigma}(s,x) \right|^{2q} \right]^\frac{1}{q} (S^b_{t-s}h)(x)^2 \mu(dx) ds\\
&\leq \sup_{(s,x) \in [0,t] \times F} \Ebb\left[ \left| \tilde{\sigma}(s,x) \right|^{2q} \right]^\frac{1}{q} \int_0^t \Vert S^b_s h \Vert_\mu^2 ds.\\
\end{split}
\end{equation*}
Now we take powers of $q$ on both sides and use \eqref{resolvint} for the required result.
\end{enumerate}
\end{proof}

\subsection{Partitions and delta-approximants}
In this section we define sequences of functions in $\Hcal$ that approximate the delta functional on $F$ in a controllable way. We then use this sequence to ``evaluate'' $U(t)$ given in \eqref{SPDEmild} at points $x \in F$. Much of the material in this section follows the method and results of \cite[Section 3.1]{Hambly2016}. 
\begin{defn}
Let $m,n \geq 0$. If $w \in \Wbb_m$ and $v \in \Wbb_n$ then we say that $w$ is a \textit{prefix} of $v$ if $m \leq n$ and $w_i = v_i$ for all $i = 1,\ldots,m$. It is a \textit{proper prefix} if in addition we have that $m < n$. These definitions also make sense if $v \in \Wbb$.
\end{defn}
We define a sequence of partitions $(\Lambda_n)_{n=0}^\infty$ on $\Wbb_*$ by
\begin{equation*}
\Lambda_n = \{ w: w = w_1 \ldots w_m \in \Wbb_*,\ r_{w_1\ldots w_{m-1}} > 2^{-n} \geq r_w \}
\end{equation*}
for $n \geq 1$, and we let $\Lambda_0 = \Wbb_0$, which is also a partition (for the definition of a partition see \cite[Definitions 1.3.9 and 1.5.6]{Kigami2001}) or \cite[Definitions 3.2 and 3.3]{Hambly2016}). These possess the following properties:
\begin{lem}\label{partitionfacts}
For each $n \geq 0$:
\begin{enumerate}
\item $|\Lambda_n| < \infty$ and $\bigcup_{w \in \Lambda_n} F_w = F$,
\item If $w \in \Lambda_n$ then there exists a subset $\Lambda' \subseteq \Lambda_{n+1}$ such that $F_w = \bigcup_{v \in \Lambda'} F_v$,
\item If $w,v \in \Lambda_n$ are distinct then $|F_w \cap F_v| < \infty$,
\item If $w \in \Lambda_n$ then $r_{\min}^{d_H} 2^{-d_H n} < \mu(F_w) \leq 2^{-d_H n}$.
\end{enumerate}
In addition, if $w \in \Wbb_*$ then there exists some $n \geq 0$ and $\Lambda' \subseteq \Lambda_n$ such that $F_w = \bigcup_{v \in \Lambda'} F_v$.
\end{lem}
\begin{proof}
\begin{enumerate}
\item Directly from definition of a partition.
\item By \cite[Lemma 3.4]{Hambly2016}.
\item By the definition of a partition combined with \cite[Proposition 1.3.5(2)]{Kigami2001}. If $w,v \in \Lambda_n$ are distinct then $F_w \cap F_v = \psi_w(F^0) \cap \psi_v(F^0)$ which is finite since $F$ is post-critically finite.
\item By definition of $\mu$ and definition of $\Lambda_n$.
\end{enumerate}
Finally we prove the last claim. The empty word is in $\Lambda_0$ so we can ignore that case. Suppose $w = w_1 \ldots w_m \in \Wbb_*$. Pick $n$ such that $r_{w_1\ldots w_{m-1}} > 2^{-n}$, then no element of $\Lambda_n$ can be a proper prefix of $w$. By the definition of the map $\pi$, for any $x \in F_w$ there exists a $w_x \in \Wbb$ which has $w$ as a prefix such that $\pi(w') = x$. By the definition of a partition there must exist a unique $v_x \in \Lambda_n$ such that $v_x$ is a prefix of $w_x$. Since $v_x$ is not a proper prefix of $w$, $w$ must be a prefix of $v_x$. This implies that $x = \pi(w_x) \in F_{v_x} \subseteq F_w$. Let $\Lambda' = \{ v \in \Lambda_n: v = v_x \text{ for some } x \in F_w \}$ be the set of $v$'s that can be obtained in this way. It follows that $F_w = \bigcup_{v \in \Lambda'} F_v$.
\end{proof}
\begin{defn}\label{nhood}
Let $n \geq 0$ and $w \in \Wbb_n$. For $x \in F$ let
\begin{equation*}
D^0_n(x) = \bigcup \{ F_w: w \in \Lambda_n,\ F_w \ni x \}
\end{equation*}
be the \textit{$n$-neighbourhood} of $x$.
\end{defn}
The following is an important structural property of $n$-neighbourhoods:
\begin{lem}\label{nhoodestim}
There exists a constant $c_2 > 0$ such that if $x,y \in F$ and $y \in D^0_n(x)$ then $R(x,y) \leq c_2 2^{-n}$.
\end{lem}
\begin{proof}
Simple corollary of \cite[Proposition 3.12]{Hambly2016}.
\end{proof}
We now define our family of functions approximating the delta functional, which is identical to those given in \cite[Definition 4.7]{Hambly2016}.
\begin{defn}
For $x \in F$ and $n \geq 0$, define
\begin{equation*}
f^x_n = \mu(D^0_n(x))^{-1} \1bb_{D^0_n(x)}.
\end{equation*}
\end{defn}
In particular we note that by Lemma \ref{partitionfacts}(4), $\Vert f^x_n \Vert_\mu^2 = \mu(D^0_n(x))^{-1} \leq r_{\min}^{-d_H} 
2^{d_H n}$, and that by Lemma \ref{nhoodestim} we have that $\lim_{n \to \infty} \langle f^x_n, g \rangle_\mu = g(x)$ for any continuous 
$g \in \Hcal$. We are now ready to prove the main result of this section.

\begin{thm}\label{randomfieldexist}
Assume either Hypothesis \ref{hyp1} or Hypothesis \ref{hyp2}. Then for each $t \in [0,T]$ and $x \in F$, the sequence 
$(\langle U(t), f^x_n \rangle_\mu)_{n \geq 0}$ is Cauchy in $L^{2p}(\Omega)$ as $n \to \infty$. Define $\tilde u(t,x)$ to be the $L^{2p}(\Omega)$-limit of $(\langle U(t), f^x_n \rangle_\mu)_{n \geq 0}$, then there exists a constant $c_3$ such that for all $(t,x) \in [0,T] \times F$ and all $n \geq 0$,
\begin{equation*}
\Ebb \left[ \left| \langle U(t), f^x_n \rangle_\mu - \tilde u(t,x) \right|^{2p} \right] \leq c_3 2^{-np}.
\end{equation*}
\end{thm}

\begin{proof}
Recall that $p \geq 1$. By \eqref{SPDEmild} and Proposition \ref{apriori}, for $m,n \geq 0$ we have that
\begin{equation*}
\begin{split}
&\Ebb \left[ \left| \langle U(t), f^x_m \rangle_\mu - \langle U(t), f^x_n \rangle_\mu \right|^{2p} \right]\\
&= \Ebb \left[ \left| \langle U(t), f^x_m - f^x_n \rangle_\mu \right|^{2p} \right]\\
&\leq 2^{2p-1} \Ebb\left[ \left| \left\langle \int_0^t S^b_{t-s} \beta_s ds, f^x_m - f^x_n \right\rangle_\mu \right|^{2p} + \left| \left\langle \int_0^t S^b_{t-s} \sigma_s dW(s), f^x_m - f^x_n \right\rangle_\mu \right|^{2p} \right]\\
&\leq c_3' \left( \int_F \int_F \rho^b_1(z,y)(f^x_m(z) - f^x_n(z)) (f^x_m(y) - f^x_n(y))\mu(dz)\mu(dy) \right)^p,
\end{split}
\end{equation*}
where $c_3' = e^{2pT}2^{2p-1} (2^{-p} + C_p) K$. By Lemma \ref{resolvLip} and Lemma \ref{nhoodestim} and the definition of $f^x_n$, we see 
that, as in the proof of \cite[Theorem 4.8]{Hambly2016}, 
\begin{equation*}
\int_F \int_F \rho^b_1(z,y)(f^x_m(z) - f^x_n(z)) (f^x_m(y) - f^x_n(y))\mu(dz)\mu(dy) \leq 4 c_2(2^{-m} + 2^{-n}).
\end{equation*}
If we let $c_3 = 4^pc_2^pc_3'$ then it follows that
\begin{equation*}
\Ebb \left[ \left| \langle U(t), f^x_m \rangle_\mu - \langle U(t), f^x_n \rangle_\mu \right|^{2p} \right] \leq c_3 (2^{-m} + 2^{-n})^p,
\end{equation*}
so $(\langle U(t), f^x_n \rangle)_\mu)_{n \geq 0}$ is Cauchy in $L^{2p}(\Omega)$ as $n \to \infty$. If we let $\tilde u(t,x)$ be the limit, then 
immediately we have that
\begin{equation*}
\Ebb \left[ \left| \langle U(t), f^x_n \rangle_\mu - \tilde u(t,x) \right|^{2p} \right] \leq c_3 2^{-np}.
\end{equation*}
\end{proof}
The collection $(\tilde u(t,x): t \in [0,T],\ x \in F)$ constructed in Theorem \ref{randomfieldexist} is the precursor to our candidate random field. 
Currently this is nothing more than a collection of unrelated real random variables indexed by $[0,T] \times F$ so we cannot show that 
$(\tilde u(t,\cdot))_{t \in [0,T]}$ is a version of $U$; there are problems of joint measurability which will be resolved in the next section.

\section{Continuity of random field}\label{cty}

In this section we collect the estimates required to use the continuity theorems \cite[Theorem 3.17, Corollary 3.19]{Hambly2016} on the 
collection $(\tilde u(t,x): t \in [0,T],\ x \in F)$.

\subsection{Spatial estimate}
\begin{prop}\label{spatial}
Assume either Hypothesis \ref{hyp1} or Hypothesis \ref{hyp2}. Then there exists a constant $c_4 > 0$ such that for all $t \in [0,T]$ and $x,y \in F$,
\begin{equation*}
\Ebb\left[ \left| \tilde u(t,x) - \tilde u(t,y) \right|^{2p} \right] \leq c_4 R(x,y)^p.
\end{equation*}
\end{prop}
\begin{proof}
By Theorem \ref{randomfieldexist}, Proposition \ref{apriori} and \eqref{SPDEmild},
\begin{equation*}
\begin{split}
\Ebb &\left[ \left| \tilde u(t,x) - \tilde u(t,y) \right|^{2p} \right] = \lim_{n \to \infty} \Ebb\left[ \left| \langle U(t), f^x_n - f^y_n \rangle_\mu \right|^{2p} \right]\\
&\leq c' \lim_{n \to \infty}\left( \int_F \int_F \rho^b_1(z_1,z_2)(f^x_n(z_1) - f^y_n(z_1)) (f^x_n(z_2) - f^y_n(z_2))\mu(dz_1)\mu(dz_2) \right)^p
\end{split}
\end{equation*}
where $c' > 0$ is independent of $x,y,t$. Since $\rho_1^b$ is jointly Lipschitz (Lemma \ref{resolvLip}) and symmetric we see by Lemma \ref{nhoodestim} and the definition of the $f^x_n$ that
\begin{equation*}
\begin{split}
\lim_{n \to \infty} &\int_F \int_F \rho^b_1(z_1,z_2)(f^x_n(z_1) - f^y_n(z_1)) (f^x_n(z_2) - f^y_n(z_2))\mu(dz_1)\mu(dz_2)\\
&= \rho_1^b(x,x) - 2\rho^b_1(x,y) + \rho_1^b(y,y)\\
&\leq 4 R(x,y).
\end{split}
\end{equation*}
The result follows.
\end{proof}
\subsection{Temporal estimate}
The temporal estimate is slightly more complicated to derive. First we prove a lemma:
\begin{lem}\label{semigpincrem}
Let $S = (S_t)_t$ be a contraction semigroup on a separable Hilbert space generated by a self-adjoint operator. If $t_0, t \in (0,\infty)$ then
\begin{equation*}
\left\Vert S_{t_0} - S_{t_0 + t}\right\Vert \leq \frac{t}{t_0 + t}.
\end{equation*}
\end{lem}
\begin{proof}
Let $L$ be the generator of $S$. This operator is non-positive and self-adjoint, and $S_t = e^{Lt}$. By the functional calculus for self-adjoint operators we thus have that
\begin{equation*}
\left\Vert S_{t_0} - S_{t_0 + t}\right\Vert \leq \sup_{\lambda \in [0,\infty)} \left( e^{-\lambda t_0} - e^{-\lambda (t_0 + t)} \right).
\end{equation*}
The function $\lambda \mapsto e^{-\lambda t_0} - e^{-\lambda (t_0 + t)}$ on $[0,\infty)$ is smooth, non-negative, bounded, and vanishes at $0$ and infinity so we differentiate to find a maximum. We find that
\begin{equation*}
\sup_{\lambda \in [0,\infty)} \left( e^{-\lambda t_0} - e^{-\lambda (t_0 + t)} \right) = \left( 1 + \frac{t}{t_0} \right)^{-\frac{t_0}{t}} \frac{t}{t_0 + t}.
\end{equation*}
Now for all $x \in (0,\infty)$ we have that $(1 + \frac{1}{x})^{-x} \leq 1$ so the result follows.
\end{proof}
\begin{prop}\label{temporal}
Assume either Hypothesis \ref{hyp1} or Hypothesis \ref{hyp2}. Then there exists a constant $c_5 > 0$ such that for all $s,t \in [0,T]$ and $x \in F$,
\begin{equation*}
\Ebb\left[ \left| \tilde u(s,x) - \tilde u(t,x) \right|^{2p} \right] \leq c_5 |s - t|^{p(d_H + 1)^{-1}}.
\end{equation*}
\end{prop}
\begin{proof}
Fix $x \in F$ and $s \in [0,T]$, $t > 0$ such that $s+t \in [0,T]$. Then by Theorem \ref{randomfieldexist}, for all $n \geq 0$ we have that
\begin{equation}\label{time1}
\Ebb\left[ \left| \tilde u(s+t,x) - \tilde u(s,x) \right|^{2p} \right] \leq 2^{2p-1}\Ebb\left[ \left| \langle U(s+t) - U(s), f^x_n \rangle_\mu \right|^{2p} \right] + c_3 2^{4p-1} \cdot 2^{-np}.
\end{equation}
Now if we let $f_n^{x*}$ be the bounded linear functional on $\Hcal$ given by $ h \mapsto \langle h , f_n^x \rangle_\mu$, then by \eqref{SPDEmild},
\begin{equation*}
\begin{split}
\langle U(s+t) - U(s), f^x_n \rangle_\mu &= \int_0^{s+t} f_n^{x*} \left( S^b_{s+t - t'} - \1bb_{[0,s]}(t') S^b_{s-t'} \right) \beta_{t'} dt'\\
&\quad + \int_0^{s+t} f_n^{x*} \left( S^b_{s+t - t'} - \1bb_{[0,s]}(t') S^b_{s-t'} \right) \sigma_{t'} dW(t').
\end{split}
\end{equation*}
We deal with these two terms separately. For the $\beta$ term, both Hypothesis \ref{hyp1} and Hypothesis \ref{hyp2} allow for the same calculation: by the self-adjointness of $S^b$ and H\"older's inequality,
\begin{equation*}
\begin{split}
&\Ebb \left[ \left| \int_0^{s+t} f_n^{x*} \left( S^b_{s+t - t'} - \1bb_{[0,s]}(t') S^b_{s-t'} \right) \beta_{t'} dt' \right|^{2p} \right]\\
&= \Ebb \left[ \left( \int_0^{s+t} \left| \left\langle \beta_{t'}, \left( S^b_{s+t - t'} - \1bb_{[0,s]}(t') S^b_{s-t'} \right) f_n^x \right\rangle_\mu \right| dt'\right)^{2p} \right]\\
&\leq \Ebb \left[ \left( \int_0^{s+t} \Vert \beta_{t'} \Vert_\mu \left\Vert \left( S^b_{s+t - t'} - \1bb_{[0,s]}(t') S^b_{s-t'} \right) f_n^x \right\Vert_\mu dt'\right)^{2p} \right]\\
&\leq \Ebb \left[ \left( \int_0^{s+t} \Vert \beta_{t'} \Vert_\mu^2 dt' \right)^p \right] \left( \int_0^{s+t} \left\Vert \left( S^b_{s+t - t'} - \1bb_{[0,s]}(t') S^b_{s-t'} \right) f_n^x \right\Vert_\mu^2 dt' \right)^p\\
&\leq \Vert f^x_n \Vert_\mu^{2p} \Ebb \left[ \left( \int_0^T \Vert \beta_{t'} \Vert_\mu^2 dt' \right)^p \right] \left( \int_0^{s+t} \left\Vert S^b_{s+t - t'} - \1bb_{[0,s]}(t') S^b_{s-t'} \right\Vert^2 dt' \right)^p
\end{split}
\end{equation*}
and then using Lemma \ref{semigpincrem},
\begin{equation}\label{semigpincremeqn}
\begin{split}
\int_0^{s+t} &\left\Vert S^b_{s+t - t'} - \1bb_{[0,s]}(t') S^b_{s-t'} \right\Vert^2 dt'\\
&= \int_0^s \left\Vert S^b_{s+t - t'} - S^b_{s-t'} \right\Vert^2 dt' + \int_s^{s+t} \left\Vert S^b_{s+t - t'} \right\Vert^2 dt'\\
&= \int_0^s \left\Vert S^b_{t + t'} - S^b_{t'} \right\Vert^2 dt' + \int_0^t \left\Vert S^b_{t'} \right\Vert^2 dt'\\
&\leq \int_0^s \frac{t^2}{(t' + t)^2} dt' + \int_0^t 1 dt'\\
&\leq 2t.
\end{split}
\end{equation}
It follows that
\begin{equation*}
\Ebb \left[\left| \int_0^{s+t} f_n^{x*} \left( S^b_{s+t - t'} - \1bb_{[0,s]}(t') S^b_{s-t'} \right) \beta_{t'} dt' \right|^{2p} \right] \leq 2^pK \Vert f^x_n \Vert_\mu^{2p} t^p.
\end{equation*}

For the $\sigma$ term, we first assume Hypothesis \ref{hyp1}. We use the Burkholder-Davis-Gundy inequality (the validity of which, given Hypothesis \ref{hyp1}, is easy to verify; see proof of Proposition \ref{apriori}(2)) and see that
\begin{equation*}
\begin{split}
&\Ebb \left[ \left|\int_0^{s+t} f_n^{x*} \left( S^b_{s+t - t'} - \1bb_{[0,s]}(t') S^b_{s-t'} \right) \sigma_{t'} dW(t') \right|^{2p} \right]^\frac{1}{p}\\
&\leq C'_p \Ebb \left[ \left(\int_0^{s+t} \left\Vert f_n^{x*} \left( S^b_{s+t - t'} - \1bb_{[0,s]}(t') S^b_{s-t'} \right) \sigma_{t'} \right\Vert_{\Lcal_2(\Hcal,\Rbb)}^2 dt' \right)^p \right]^\frac{1}{p}\\
&= C'_p \Ebb \left[ \left(\int_0^{s+t} \left\Vert \sigma^*_{t'} \left( S^b_{s+t - t'} - \1bb_{[0,s]}(t') S^b_{s-t'} \right) f^x_n \right\Vert_\mu^2 dt' \right)^p \right]^\frac{1}{p}\\
&\leq C'_p \int_0^{s+t} \Ebb \left[ \left\Vert \sigma^*_{t'} \right\Vert^{2p} \right]^\frac{1}{p} \left\Vert \left( S^b_{s+t - t'} - \1bb_{[0,s]}(t') S^b_{s-t'} \right) f^x_n \right\Vert_\mu^2 dt'\\
&\leq C'_p \sup_{t' \in [0,T]} \Ebb \left[ \Vert \sigma_{t'} \Vert^{2p} \right]^\frac{1}{p} \int_0^{s+t} \left\Vert \left( S^b_{s+t - t'} - \1bb_{[0,s]}(t') S^b_{s-t'} \right) f^x_n \right\Vert_\mu^2 dt' \\
&\leq C'_p K^\frac{1}{p} \Vert f^x_n \Vert_\mu^2 \int_0^{s+t} \left\Vert S^b_{s+t - t'} - \1bb_{[0,s]}(t') S^b_{s-t'} \right\Vert^2 dt'\\
&\leq 2 C'_p K^\frac{1}{p} \Vert f^x_n \Vert_\mu^2 t
\end{split}
\end{equation*}
where we have used the Minkowski integral inequality and the calculation \eqref{semigpincremeqn} a second time. If we instead assume Hypothesis \ref{hyp2} then we get the same estimate by a similar calculation, but use the Minkowski integral inequality twice; see the proof of Proposition \ref{apriori}. Plugging these values into \eqref{time1} we get
\begin{equation*}
\Ebb\left[ \left| \tilde u(s+t,x) - \tilde u(s,x) \right|^{2p} \right] \leq 2^{5p-2}  (1 + (C'_p)^p) K \Vert f^x_n \Vert_\mu^{2p} t^p + c_3 2^{4p-1} \cdot 2^{-np}.
\end{equation*}
We now use the fact that $\Vert f^x_n \Vert_\mu^2 = \mu(D^0_n(x))^{-1} \leq r_{\min}^{-d_H} 2^{d_H n}$. If we let $c'_1 = 2^{5p-2}  (1 + (C'_p)^{2p}) K r_{\min}^{-pd_H}$ and $c'_2 = c_3 2^{4p-1}$ we have that
\begin{equation*}
\Ebb\left[ \left| \tilde u(s+t,x) - \tilde u(s,x) \right|^{2p} \right] \leq c'_1 2^{d_H np}t^p + c'_2 2^{-np}.
\end{equation*}
The final step is to minimize the right-hand side of the above equation over $n \geq 0$, which is similar to the method used to prove \cite[Proposition 5.5]{Hambly2016}. It is in fact easier to let $c''_2 = c'_2 \vee d_Hc'_1 T^p$ and then minimize $c'_1 2^{d_H np}t^p + c''_2 2^{-np}$ over $n \geq 0$. To this end, let $f(y) = c'_1 2^{d_H y}t^p + c''_2 2^{-y}$ be a function on $\Rbb$, then the unique minimum value of $f$ occurs at
\begin{equation*}
y_0 := \frac{1}{(d_H + 1) \log 2} \log\left( \frac{c_2''}{d_Hc'_1 t^p} \right) \geq 0.
\end{equation*}
If we set $n = \lceil y_0 p^{-1} \rceil$ then since $f$ is increasing in $[y_0, \infty)$ we get
\begin{equation*}
\begin{split}
\Ebb\left[ \left| \tilde u(s+t,x) - \tilde u(s,x) \right|^{2p} \right] &\leq f(np)\\
&\leq f(y_0 + p)\\
&= c'_1 2^{d_Hp} \left( \frac{c_2''}{d_Hc'_1 t^p} \right)^{\frac{d_H}{d_H + 1}}t^p + c''_2 2^{-p} \left( \frac{c_2''}{d_Hc'_1 t^p} \right)^{-\frac{1}{d_H + 1}}\\
&=: c''_1 t^{\frac{p}{d_H + 1}}
\end{split}
\end{equation*}
where $c''_1 > 0$ is a constant that does not depend on $s,t,x$.
\end{proof}

\subsection{Proof of main theorem}

\begin{proof}[Proof of Theorem \ref{mainthm}]
The spatial and temporal estimates in Proposition \ref{spatial} and Proposition \ref{temporal} respectively combined with the continuity theorem \cite[Corollary 3.19]{Hambly2016} together imply that there exists a version $u: \Omega \times [0,T] \times F \to \Rbb$ of the collection $(\tilde u(t,x) : t \in [0,T],\ x \in F)$  satisfying the required H\"{o}lder continuity properties. We have that $u(t,x)$ is a random variable satisfying $\tilde u(t,x) = u(t,x)$ almost surely for each $(t,x) \in [0,T] \times F$. It only remains to show that $(u(t,\cdot))_{t \in [0,T]}$ is a version of $U$.

We need to show that for each $t \in [0,T]$, $u(t,\cdot) = U(t)$ as elements of $\Hcal$ almost surely. Fix $t \in [0,T]$. Since $u(t,\cdot)$ is almost surely continuous on the compact metric space $(F,R)$ it is a Carathéodory function, so by \cite[Lemma 4.51]{Aliprantis2006} it must be jointly measurable on $\Omega \times F$. Additionally the compactness of $F$ implies that $u(t,\cdot)$ is almost surely bounded, so we have that $u(t,\cdot) \in \Hcal$ almost surely. Since $\Dcal$ is dense in $\Hcal$, it will suffice to show that $\langle u(t,\cdot), h \rangle_\mu = \langle U(t), h \rangle_\mu$ for all $h \in \Dcal$ almost surely. Furthermore, every element of $\Dcal$ is continuous, and thus can be approximated arbitrarily closely in $\Hcal$ by finite linear combinations of functions of the form $\1bb_{F_w}$, $w \in \Wbb_*$. Since $\Wbb_*$ is countable, we therefore need only to show that $\langle u(t,\cdot) , \1bb_{F_w} \rangle_\mu = \langle U(t) , \1bb_{F_w} \rangle_\mu$ almost surely for each $w \in \Wbb_*$. To this end, we fix $w \in \Wbb_*$. By the final part of Lemma \ref{partitionfacts} and then repeated application of Lemma \ref{partitionfacts}(2), there exists $n_w \geq 0$ such that for each $n \geq n_w$, there exists a subset $\Lambda_n' \subseteq \Lambda_n$ such that
\begin{equation*}
\bigcup_{v \in \Lambda_n'} F_v = F_w.
\end{equation*}
By Lemma \ref{partitionfacts}(3), for each $v \in \Lambda'_n$ we have that $D^0_n(x) = F_v$ for all but finitely many $x \in F_v$, so the map $x \mapsto f^x_n$ from $F_w$ to $\Hcal$ is measurable. Thus $(\omega, x) \mapsto \langle U(t), f^x_n \rangle_\mu \1bb_{F_w}(x)$ is jointly measurable. Again by Lemma \ref{partitionfacts}(3), for each $n \geq n_w$ the set of $y \in F$ such that $y \in F_v \cap F_{v'}$ for two distinct $v, v' \in \Lambda'_n$ is finite. It follows that
\begin{equation*}
\int_{F_w} \langle U(t), f^x_n \rangle_\mu \mu(dx) = \sum_{v \in \Lambda_n'} \int_{F_v} \langle U(t), f^x_n \rangle_\mu \mu(dx).
\end{equation*}
Thus we have that
\begin{equation*}
\begin{split}
\sum_{v \in \Lambda_n'} \int_{F_v} \langle U(t), f^x_n \rangle_\mu \mu(dx) &= \sum_{v \in \Lambda_n'} \mu(F_v)^{-1} \int_{F_v} \langle U(t), \1bb_{F_v} \rangle_\mu \mu(dx)\\
&= \sum_{v \in \Lambda_n'} \langle U(t), \1bb_{F_v} \rangle_\mu\\
&= \sum_{v \in \Lambda_n'}\int_{F_v} U(t)(y) \mu(dy)\\
&= \int_{F_w} U(t)(y) \mu(dy).
\end{split}
\end{equation*}
We conclude that
\begin{equation*}
\int_{F_w} \langle U(t), f^x_n \rangle_\mu \mu(dx) = \langle U(t), \1bb_{F_w} \rangle_\mu
\end{equation*}
for all $n \geq n_w$. Finally from Jensen's inequality, Tonelli's theorem and Theorem \ref{randomfieldexist} we have that
\begin{equation*}
\begin{split}
\Ebb &\left[ \left| \int_{F_w} \langle U(t), f^x_n \rangle_\mu \mu(dx) - \langle u(t,\cdot) , \1bb_{F_w} \rangle_\mu \right|^{2p} \right]\\
&\leq \mu(F_w)^{2p-1} \int_{F_w} \Ebb\left[ \left| \langle U(t), f^x_n \rangle_\mu - u(t,x) \right|^{2p} \right] \mu(dx)\\
&\leq c_3 \mu(F_w)^{2p} 2^{-np}
\end{split}
\end{equation*}
for $n \geq n_w$. Then using Markov's inequality, for every $\epsilon > 0$ we have the bound
\begin{equation*}
\Pbb\left[ \left| \int_{F_w} \langle U(t), f^x_n \rangle_\mu \mu(dx) - \langle u(t,\cdot) , \1bb_{F_w} \rangle_\mu \right| \geq \epsilon \right] \leq c_3 \mu(F_w)^{2p} \epsilon^{-2p} 2^{-np}
\end{equation*}
for sufficiently large $n$. Therefore by the Borel-Cantelli lemma,
\begin{equation*}
\int_{F_w} \langle U(t), f^x_n \rangle_\mu \mu(dx) \to \langle u(t,\cdot) , \1bb_{F_w} \rangle_\mu
\end{equation*}
as $n \to \infty$ almost surely. So $\langle u(t,\cdot) , \1bb_{F_w} \rangle_\mu = \langle U(t) , \1bb_{F_w} \rangle_\mu$ almost surely and the proof is complete.
\end{proof}

\section{Application to a class of da Prato--Zabczyk SPDEs}\label{applications1}
In this section we give an example of how to apply Theorem \ref{mainthm} to the solutions of a class of SPDEs on p.c.f. fractals of the form
\begin{equation}\label{SPDEapp}
\begin{split}
dY(t) &= \Delta_bY(t)dt + f(t,Y(t))dt + g(t,Y(t))dW(t), \quad t \in [0,T],\\
Y(0) &= Y_0
\end{split}
\end{equation}
with $T > 0$ and $b \in 2^{F^0}$. We recall that a \textit{mild solution} of the SPDE \eqref{SPDEapp} is a predictable $\Hcal$-valued process $Y = (Y_t)_{t \in [0,T]}$ satisfying
\begin{equation*}
Y(t) = S^b_tY_0 + \int_0^t S^b_{t-s}f(s,Y(s))ds + \int_0^t S^b_{t-s} g(s,Y(s))dW(s)
\end{equation*}
almost surely for each $t \in [0,T]$. For a topological space $\Scal$, let $\Bcal(\Scal)$ be the Borel $\sigma$-algebra on $\Scal$. Let $\Pcal_T$ be the predictable $\sigma$-algebra on $\Omega \times [0,T]$ associated with the truncated filtration $(\Fcal_t: 0 \leq t \leq T)$. The following hypothesis is adapted from \cite[Hypothesis 7.2]{DaPrato1992}:
\begin{hyp}\label{apphyp}
We make the following assumptions for the SPDE \eqref{SPDEapp}. Suppose there exists $q \geq 2$ such that the following hold:
\begin{enumerate}
\item $Y_0$ is an $\Hcal$-valued $\Fcal_0$-measurable random variable such that $\Ebb[\Vert Y_0 \Vert_\mu^q] < \infty$.
\item The function $f: \Omega \times [0,T] \times \Hcal \to \Hcal$ is measurable from $\Pcal_T \otimes \Bcal(\Hcal)$ into $\Bcal(\Hcal)$.
\item The function $g: \Omega \times [0,T] \times \Hcal \to \Lcal(\Hcal)$ is measurable from $\Pcal_T \otimes \Bcal(\Hcal)$ into $\Bcal(\Lcal(\Hcal))$.
\item There exists a constant $C>0$ and a non-negative real predictable process $M: \Omega \times [0,T] \to \Rbb$ with
\begin{equation*}
\Vert M \Vert_q := \sup_{s \in [0,T]} \Ebb\left[ M(s)^q \right]^\frac{1}{q} < \infty
\end{equation*}
such that for all $h,h' \in \Hcal$, $\omega \in \Omega$ and $t \in [0,T]$ we have that
\begin{equation*}
\Vert f(\omega,t,h) \Vert_\mu + \Vert g(\omega,t,h) \Vert \leq M(\omega,t) + C\Vert h \Vert_\mu
\end{equation*}
and
\begin{equation*}
\Vert f(\omega,t,h) - f(\omega,t,h') \Vert_\mu + \Vert g(\omega,t,h) - g(\omega,t,h') \Vert \leq C\Vert h - h' \Vert_\mu.
\end{equation*}
\end{enumerate}
\end{hyp}
\begin{defn}
As in the proof of \cite[Theorem 7.2]{DaPrato1992}, let $\Hscr_q$ be the space of $\Hcal$-valued predictable processes $Z = (Z(t))_{t \in [0,T]}$ such that
\begin{equation*}
\vertiii{Z}_q := \sup_{t \in [0,T]} \Ebb \left[ \Vert Z(t) \Vert_\mu^q \right]^\frac{1}{q} < \infty.
\end{equation*}
We equip $\Hscr_q$ with the norm $\vertiii{\cdot}_q$, which makes it a Banach space. Define a stochastic-process-valued function $\Kscr$ on $\Hscr_q$ by
\begin{equation*}
\Kscr(Z)(t) = S^b_tY_0 + \int_0^t S^b_{t-s}f(s,Z(s))ds + \int_0^t S^b_{t-s} g(s,Z(s))dW(s).
\end{equation*}
\end{defn}
\begin{lem}\label{welldefK}
Assume Hypothesis \ref{apphyp}. Then there exist constants $\alpha_1, \alpha_2 > 0$ dependent only on $T$ such that if $Z \in \Hscr_q$ then
\begin{equation*}
\vertiii{\Kscr(Z)}_q \leq \Ebb[\Vert Y_0 \Vert_\mu^q]^\frac{1}{q} + \alpha_1 + \alpha_2 \vertiii{Z}_q.
\end{equation*}
In particular, the function $\Kscr$ maps $\Hscr_q$ into itself.
\end{lem}
\begin{proof}
Fix $Z \in \Hscr_q$. For $t \in [0,T]$ we see that
\begin{equation*}
\begin{split}
\Ebb \left[ \left\Vert \int_0^t S^b_{t-s}f(s,Z(s))ds \right\Vert_\mu^q \right] &\leq \Ebb \left[ \left| \int_0^t \Vert f(s,Z(s)) \Vert_\mu ds \right|^q \right]\\
&\leq t^{q-1} \Ebb \left[ \int_0^t \Vert f(s,Z(s)) \Vert_\mu^q ds \right]\\
&\leq t^{q-1} \Ebb \left[ \int_0^t (M(s) + C\Vert Z(s) \Vert_\mu)^q ds \right]\\
&\leq 2^{q-1}t^q (\Vert M \Vert_q^q + C^q\vertiii{Z}_q^q)
\end{split}
\end{equation*}
and by \cite[Theorem 4.37]{DaPrato1992} there exists a constant $k_q > 0$ such that
\begin{equation*}
\begin{split}
\Ebb &\left[ \left\Vert \int_0^t S^b_{t-s} g(s,Z(s))dW(s) \right\Vert_\mu^q \right]\\
&\leq k_q \left( \int_0^t \Ebb \left[ \Vert S^b_{t-s} g(s,Z(s)) \Vert_{\Lcal_2(\Hcal,\Hcal)}^q \right]^\frac{2}{q} ds \right)^\frac{q}{2}\\
&\leq k_q \left( \int_0^t \Ebb \left[ \Vert g(s,Z(s)) \Vert^q \right]^\frac{2}{q} \Vert S^b_{t-s} \Vert_{\Lcal_2(\Hcal,\Hcal)}^2 ds \right)^\frac{q}{2}\\
&\leq k_q \sup_{s \in [0,t]} \Ebb \left[ \Vert g(s,Z(s)) \Vert^q \right] \left( \int_0^t \Vert S^b_s \Vert_{\Lcal_2(\Hcal,\Hcal)}^2 ds \right)^\frac{q}{2}\\
&\leq 2^{q - 1} k_q (\Vert M \Vert_q^q + C^q\vertiii{Z}_q^q) \left( \int_0^t \Vert S^b_s \Vert_{\Lcal_2(\Hcal,\Hcal)}^2 ds \right)^\frac{q}{2}\\
&< \infty.
\end{split}
\end{equation*}
The integral of $\Vert S^b_s \Vert_{\Lcal_2(\Hcal,\Hcal)}^2$ is finite by Proposition \ref{spectra} and the fact that $d_s < 2$. Finally we note that $\Vert S^b_tY_0 \Vert_\mu \leq \Vert Y_0 \Vert_\mu < \infty$ and we put together all of these estimates to see that there exist constants $\alpha_1, \alpha_2 > 0$ dependent only on $T$ such that
\begin{equation*}
\vertiii{\Kscr(Z)}_q \leq \Ebb[\Vert Y_0 \Vert_\mu^q]^\frac{1}{q} + \alpha_1 + \alpha_2 \vertiii{Z}_q.
\end{equation*}

It remains to show that $\Kscr(Z)$ is predictable (that is to say, it has a predictable version). The term $S^b_t Y_0$ causes no trouble as it is $\Fcal_0$-measurable. By \cite[Proposition 3.7(ii)]{DaPrato1992}, it is enough to show that the convolutions $t \mapsto\int_0^t S^b_{t-s}f(s,Z(s))ds$ and $t \mapsto \int_0^t S^b_{t-s} g(s,Z(s))dW(s)$ are stochastically continuous. These follow in a similar way to the proof of \cite[Theorem 5.2(i)]{DaPrato1992}. We treat the $g$ integral first as it is the more difficult of the two and is instructive. We in fact prove $L^2$ stochastic continuity. For $0 \leq t' < t \leq T$ we have by It\=o's isometry:
\begin{equation*}
\begin{split}
\Ebb &\left[ \left( \int_0^t S^b_{t-s} g(s,Z(s))dW(s) - \int_0^{t'} S^b_{t'-s} g(s,Z(s))dW(s) \right)^2 \right]\\
&= \Ebb \left[ \int_0^t \left\Vert \left( S^b_{t-s} - \1bb_{[0,t']}(s) S^b_{t'-s} \right) g(s,Z(s)) \right\Vert_{\Lcal_2(\Hcal,\Hcal)}^2 ds \right]\\
&\leq \int_0^t \left\Vert S^b_{t-s} - \1bb_{[0,t']}(s) S^b_{t'-s} \right\Vert_{\Lcal_2(\Hcal,\Hcal)}^2 \Ebb \left[ \Vert g(s,Z(s)) \Vert^2 \right] ds\\
&\leq 2(\Vert M \Vert_2^2 + C^2\vertiii{Z}_2^2) \int_0^t \left\Vert S^b_{t-s} - \1bb_{[0,t']}(s) S^b_{t'-s} \right\Vert_{\Lcal_2(\Hcal,\Hcal)}^2 ds,\\
\end{split}
\end{equation*}
and then
\begin{equation*}
\begin{split}
\int_0^t &\left\Vert S^b_{t-s} - \1bb_{[0,t']}(s) S^b_{t'-s} \right\Vert_{\Lcal_2(\Hcal,\Hcal)}^2 ds\\
&= \int_0^{t'} \left\Vert S^b_{t-s} - S^b_{t'-s} \right\Vert_{\Lcal_2(\Hcal,\Hcal)}^2 ds + \int_{t'}^t \left\Vert S^b_{t-s} \right\Vert_{\Lcal_2(\Hcal,\Hcal)}^2 ds\\
&= \int_0^{t'} \left\Vert S^b_{s + t - t'} - S^b_s \right\Vert_{\Lcal_2(\Hcal,\Hcal)}^2 ds + \int_0^{t - t'} \left\Vert S^b_s \right\Vert_{\Lcal_2(\Hcal,\Hcal)}^2 ds\\
&\leq \sum_{k=1}^\infty \int_0^T e^{-2\lambda^b_k s} \left( e^{-\lambda^b_k(t - t')} - 1 \right)^2 ds + \sum_{k=1}^\infty \int_0^{t - t'} e^{-2\lambda^b_k s} ds\\
\end{split}
\end{equation*}
which tends to $0$ as $t - t' \searrow 0$ by the dominated convergence theorem and Proposition \ref{spectra}. Thus we have proven $L^2$ stochastic continuity for the $g$ integral. Now $L^2$ stochastic continuity for the $f$ integral follows very similarly; we may even use the fact that $\Vert \Acal \Vert \leq \Vert \Acal \Vert_{\Lcal_2(\Hcal,\Hcal)}$ for all $\Acal \in \Lcal_2(\Hcal,\Hcal)$ to end up doing exactly the same calculation as the one above.
\end{proof}
\begin{thm}
Assume Hypothesis \ref{apphyp}. Then the SPDE \eqref{SPDEapp} has a unique mild solution $Y$ in $\Hscr_q$.

Let $U(t) = Y(t) - S^b_tY_0$ for $t \in [0,T]$. If $q > 2(d_H + 1)^2$ then $U$ has a version which is a continuous random field $u: \Omega \times [0,T] \times F \to \Rbb$ and has the following H\"{o}lder exponents:
\begin{enumerate}
\item $u$ is almost surely essentially $\left( \frac{1}{2} (d_H + 1)^{-1} - \frac{d_H + 1}{q} \right)$-H\"{o}lder continuous in $[0,T] \times F$ with respect to $R_\infty$,
\item For each $t \in [0,T]$, $u(t,\cdot)$ is almost surely essentially $\left( \frac{1}{2} - \frac{d_H}{q} \right)$-H\"{o}lder continuous in $F$ with respect to $R$,
\item For each $x \in F$, $u(\cdot,x)$ is almost surely essentially $\left( \frac{1}{2} (d_H + 1)^{-1} - \frac{1}{q} \right)$-H\"{o}lder continuous in $[0,T]$.
\end{enumerate}
\end{thm}
\begin{proof}
For $Z_1, Z_2 \in \Hscr_q$, we see that
\begin{equation*}
\begin{split}
&\Kscr(Z_1)(t) - \Kscr(Z_2)(t)\\
&= \int_0^t S^b_{t-s} \left(f(s,Z_1(s)) - f(s,Z_2(s)) \right)ds + \int_0^t S^b_{t-s} \left(g(s,Z_1(s)) - g(s,Z_2(s)) \right) dW(s).
\end{split}
\end{equation*}
We proceed in  a similar way as in the proof of Lemma \ref{welldefK}:
\begin{equation*}
\begin{split}
\Ebb &\left[ \left\Vert \int_0^t S^b_{t-s} \left(f(s,Z_1(s)) - f(s,Z_2(s)) \right) ds \right\Vert_\mu^q \right]\\
&\leq t^{q-1} \Ebb \left[ \int_0^t \Vert f(s,Z_1(s)) - f(s,Z_2(s)) \Vert_\mu^q ds \right]\\
&\leq 2^{q-1}t^q C^q \vertiii{Z_1 - Z_2}_q^q
\end{split}
\end{equation*}
and by \cite[Theorem 4.37]{DaPrato1992} there exists a constant $k_q > 0$ such that
\begin{equation*}
\begin{split}
\Ebb& \left[ \left\Vert \int_0^t S^b_{t-s} \left(g(s,Z_1(s)) - g(s,Z_2(s)) \right)dW(s) \right\Vert_\mu^q \right]\\
&\leq k_q \left( \int_0^t \Ebb\left[ \Vert S^b_{t-s} \left(g(s,Z_1(s)) - g(s,Z_2(s)) \right) \Vert_{\Lcal_2(\Hcal,\Hcal)}^q \right]^\frac{2}{q} ds \right)^\frac{q}{2}\\
&\leq k_q C^q \left( \int_0^t \Ebb\left[ \Vert Z_1(s) - Z_2(s) \Vert_\mu^p \right]^\frac{2}{q} \Vert S^b_{t-s} \Vert_{\Lcal_2(\Hcal,\Hcal)}^2 ds \right)^\frac{q}{2}\\
&\leq k_q C^q \vertiii{Z_1 - Z_2}_q^q \left( \int_0^t \Vert S^b_s \Vert_{\Lcal_2(\Hcal,\Hcal)}^2 ds \right)^\frac{q}{2}.
\end{split}
\end{equation*}
Now $\int_0^t \Vert S^b_s \Vert_{\Lcal_2(\Hcal,\Hcal)}^2 ds = \sum_{i=1}^\infty \int_0^t e^{-2\lambda^b_i s} ds < \infty$ for all $t \in [0,\infty)$, see the proof of Lemma \ref{welldefK}. Therefore by the dominated convergence theorem, the map $t \mapsto \int_0^t \Vert S^b_s \Vert_{\Lcal_2(\Hcal,\Hcal)}^2 ds$ is continuous with $\lim_{t \to 0} \int_0^t \Vert S^b_s \Vert_{\Lcal_2(\Hcal,\Hcal)}^2 ds = 0$. We thus see that there exists an increasing continuous function $\alpha_3: [0,\infty) \to [0,\infty)$ independent of $Y_0$ with $\alpha_3(0) = 0$ and such that
\begin{equation*}
\vertiii{\Kscr(Z_1) - \Kscr(Z_2)}_q \leq \alpha_3(T) \vertiii{Z_1 - Z_2}_q
\end{equation*}
for all $Z_1, Z_2 \in \Hscr_q$. Therefore using Lemma \ref{welldefK}, if $\alpha_3(T) < 1$ then $\Kscr$ is a strict contraction on $\Hscr_q$ and so the SPDE \eqref{SPDEapp} has a unique mild solution in $[0,T]$ by the contraction mapping theorem. If this condition does not hold we simply take some $\tilde T > 0$ satisfying $\alpha_3(\tilde T) < 1$ and solve the SPDE in the intervals $[0,\tilde T]$, $[\tilde T, 2\tilde T]$, and so on (this is possible by Lemma \ref{welldefK} and the independence of $\alpha_3$ from $Y_0$). We have thus proven existence of a unique mild solution $Y$ to \eqref{SPDEapp} in $\Hscr_q$.

Now we assume $q > 2(d_H + 1)^2$ and set $U = Y - S^bY_0$. For $t \in [0,T]$, let $\beta_t = f(t,Y(t))$ and $\sigma_t = g(t,Y(t))$. Then we see that \eqref{SPDEmild} is satisfied. The solution process $Y$ is predictable, so $\beta$ and $\sigma$ immediately have the required measurability properties and so it remains to show that Hypothesis \ref{hyp1} holds. We have that
\begin{equation*}
\begin{split}
\Ebb\left[ \left( \int_0^T \Vert f(s,Y(s)) \Vert_\mu^2 ds \right)^\frac{q}{2} \right] &\leq \Ebb\left[ \left( \int_0^T (M(s) + C\Vert Y(s) \Vert_\mu)^2 ds \right)^\frac{q}{2} \right]\\
&\leq T^q \sup_{s \in [0,T]} \Ebb\left[(M(s) + C\Vert Y(s) \Vert_\mu)^q \right]\\
&< \infty
\end{split}
\end{equation*}
and
\begin{equation*}
\sup_{s \in [0,T]} \Ebb\left[ \Vert g(s,Y(s)) \Vert^q \right] \leq \sup_{s \in [0,T]} \Ebb\left[(M(s) + C\Vert Y(s) \Vert_\mu)^q \right] < \infty
\end{equation*}
so Theorem \ref{mainthm} can be used with $p = \frac{q}{2}$.
\end{proof}

\section{Application to a class of Walsh SPDEs}\label{applications2}

We consider the SPDE
\begin{equation}\label{WSPDE}
\begin{split}
\frac{\partial u}{\partial t}(t,x) &= \Delta_b u(t,x) + f(t,u(t,x)) + g(t,u(t,x))\xi(t,x),\\
u(0,x) &= u_0(x)
\end{split}
\end{equation}
for $(t,x) \in [0,T] \times F$, where $T > 0$ and $b \in 2^{F^0}$, $u_0: \Omega \times F \to \Rbb$, and $f,g$ are functions from $\Omega \times [0,T] 
\times F$ to $\Rbb$. We take $\xi$ to be an $\Fbb$-space-time white noise on $(F,\mu)$ (in the martingale measure sense of \cite{Walsh1986}). 
Without loss of generality we may assume that $\xi$ is the space-time white noise associated with the cylindrical Wiener process $W$ considered 
previously; that is, for all $h \in \Hcal$ and $t \geq 0$,
\begin{equation*}
\int_0^t \int_F h(y) \xi(s,y) \mu(dy) ds = \langle h, W(t) \rangle_\mu.
\end{equation*}
Recall that $\Pcal_T$ is the predictable $\sigma$-algebra on $\Omega \times [0,T]$ associated with the truncated filtration $(\Fcal_t: 0 \leq t \leq T)$. 
We now define the spaces in which we will look for solutions to \eqref{WSPDE}.
\begin{defn}
Let $q \geq 2$. Let $\Scal_q$ be the space of processes $v = \{ v(x): x \in F \}$ such that $v: \Omega \times F \to \Rbb$ is measurable from 
$\Fcal_0 \otimes \Bcal(F)$ into $\Bcal(\Rbb)$ and
\begin{equation*}
\Vert v \Vert_q := \sup_{x \in F} \Ebb\left[ |v(x)|^q \right]^\frac{1}{q} < \infty.
\end{equation*}
This can be shown to be a Banach space with the norm $\Vert \cdot \Vert_q$, if we identify processes $v_1, v_2$ such that $v_1(x) = v_2(x)$ almost 
surely for all $x \in F$. Evidently $\Scal_{q_1} \subseteq \Scal_{q_2}$ if $q_1 \geq q_2 \geq 2$.

Likewise for $T > 0$ let $\Scal_{q,T}$ be the space of processes $v = \{ v(t,x): (t,x) \in [0,T] \times F \}$ such that $v$ is predictable (which means 
that $v: \Omega \times [0,T] \times F \to \Rbb$ is measurable from $\Pcal_T \otimes \Bcal(F)$ into $\Bcal(\Rbb)$, see \cite{Walsh1986}) and such that
\begin{equation*}
\Vert v \Vert_{q,T} := \sup_{t \in [0,T]}\sup_{x \in F} \Ebb\left[ |v(t,x)|^q \right]^\frac{1}{q} < \infty.
\end{equation*}
This is likewise a Banach space with the norm $\Vert \cdot \Vert_{q,T}$, if we identify processes $v_1, v_2$ such that $v_1(t,x) = v_2(t,x)$ almost 
surely for all $(t,x) \in [0,T] \times F$. As before, $\Scal_{q_1,T} \subseteq \Scal_{q_2,T}$ if $q_1 \geq q_2 \geq 2$.
\end{defn}

\begin{hyp}\label{walshhyp}
We make the following assumptions. Suppose there exists $q \geq 2$ such that:
\begin{enumerate}
\item $u_0 \in \Scal_q$.
\item $f,g: \Omega \times [0,T] \times \Rbb \to \Rbb$ are functions which are measurable from $\Pcal_T \otimes \Bcal(\Rbb)$ into $\Bcal(\Rbb)$ and 
obey the following Lipschitz and linear growth conditions: There exists a constant $C > 0$ and a non-negative real predictable process $M: \Omega 
\times [0,T] \to \Rbb$ with
\begin{equation*}
\Vert M \Vert_{q,T} := \sup_{s \in [0,T]} \Ebb\left[ M(s)^q \right]^\frac{1}{q} < \infty
\end{equation*}
such that for all $(\omega,t) \in \Omega \times [0,T]$ and all $x,y \in \Rbb$,
\begin{equation*}
\begin{split}
|f(\omega,t,x) - f(\omega,t,y)| + |g(\omega,t,x) - g(\omega,t,y)| &\leq C|x - y|,\\
|f(\omega,t,x)| + |g(\omega,t,x)| &\leq M(\omega,t) + C|x|.
\end{split}
\end{equation*}
\end{enumerate}
\end{hyp}

We usually suppress the dependence of $f$, $g$ and $M$ on $\omega$. Note that we use the same notation $\Vert \cdot \Vert_{q,T}$ for $M$ and 
for elements of $\Scal_{q,T}$, but the meaning will be clear from context. For the sake of the following definitions we now give a (suboptimal) 
technical result on the growth of the eigenfunctions $\phi^b_k$. Compare \cite[Theorem 4.5.4]{Kigami2001}.

\begin{lem}\label{thm:efuncgrow}
If $b \in 2^{F^0}$, then $\sup_{x \in F}|\phi^b_k(x)| < \infty$ for all $k \geq 1$ and
\begin{equation*}
\sup_{x \in F}|\phi^b_k(x)| = O(k^\frac{1}{d_s})
\end{equation*}
as $k \to \infty$.
\end{lem}

\begin{proof}
By \cite[Proposition 7.16(b)]{Barlow1998}, there exists $c > 0$ such that, if $b \in 2^{F^0}$ and $k \geq 1$, then for all $x \in F$,
\begin{equation*}
\phi^b_k(x)^2 \leq 2 + c \lambda^b_k.
\end{equation*}
Then Proposition \ref{spectra} implies the required result.
\end{proof}
\begin{defn}[Heat kernel]
Let $b \in 2^{F^0}$. For $(t,x,y) \in (0,\infty) \times F \times F$ let
\begin{equation*}
p^b_t(x,y) = \sum_{k=1}^\infty e^{-\lambda^b_k t} \phi^b_k(x) \phi^b_k(y).
\end{equation*}
This is non-negative and jointly continuous in $(0,\infty) \times F \times F$; see \cite[Definition A.2.12]{Kigami2001}. It is called the \textit{heat kernel} associated with the boundary condition $b$. Note the obvious symmetry in $(x,y)$. Observe that if $h \in \Hcal$ and $t,x \in (0,\infty) \times F$ then
\begin{equation*}
\int_F p^b_t(x,y) h(y) \mu(dy) = \langle p^b_t(x,\cdot) , h \rangle_\mu = S^b_t h(x),
\end{equation*}
which in particular implies that $p^b$ is the transition density of the diffusion $X^b$. Due to the above identity, for $t = 0$ we define
\begin{equation*}
\int_F p^b_0(x,y) h(y) \mu(dy) := h(x)
\end{equation*}
(so long as $h(x)$ is well-defined). Thus if $s \geq 0$ and $t > 0$ and $x,y \in F$ then we see that
\begin{equation*}
\int_F p^b_s(x,z) p^b_t(z,y) \mu(dz) = p^b_{s+t}(x,y).
\end{equation*}
\end{defn}
\begin{defn}
A \textit{mild solution} of the SPDE \eqref{WSPDE} is a predictable process $\{ u(t,x): (t,x) \in [0,\infty) \times F \}$ such that for each $(t,x) \in [0,\infty) \times F$ we have that
\begin{equation}\label{WSPDEmild}
\begin{split}
u(t,x) = &\int_F p^b_t(x,y) u_0(y) \mu(dy) + \int_0^t \int_F p^b_{t-s}(x,y) f(s,u(s,y)) \mu(dy)ds\\
&+ \int_0^t \int_F p^b_{t-s}(x,y) g(s,u(s,y)) \xi(s,y) \mu(dy)ds
\end{split}
\end{equation}
almost surely.
\end{defn}
\subsection{Existence and uniqueness}
We begin with a number of estimates that will be useful in the current and subsequent sections.
\begin{lem}\label{thm:hkestim}
The following estimates on the heat kernel $p^b$ hold for any $T > 0$:
\begin{enumerate}
\item There exists a constant $c_1(T) > 0$ such that for all $(t,x,y) \in (0,T] \times F \times F$ and any $b \in 2^{F^0}$,
\begin{equation*}
0 \leq p^b_t(x,y) \leq c_1(T) t^{-\frac{d_s}{2}}.
\end{equation*}
\item There exists a constant $c_2(T) > 0$ such that for all $(t,x,x',y) \in (0,T] \times F \times F \times F$ and any $b \in 2^{F^0}$,
\begin{equation*}
\left| p^b_t(x,y) - p^b_t(x',y) \right|^2 \leq c_2(T) R(x,x') t^{-1-\frac{d_s}{2}}.
\end{equation*}
\item There exists a constant $c_3(T) > 0$ such that for all $(s,t,x) \in (0,T] \times (0,T] \times F$ with $s \leq t$ and any $b \in 2^{F^0}$,
\begin{equation*}
\left| p^b_s(x,x) - p^b_t(x,x) \right| \leq c_3(T) \left( s^{-\frac{d_s}{2}} - t^{-\frac{d_s}{2}} \right).
\end{equation*}
\end{enumerate}
\end{lem}
\begin{proof}
\begin{enumerate}
\item By \cite[Theorem A.2.16]{Kigami2001} we see that $0 \leq p^b_t(x,y) \leq p^N_t(x,y)$ for all $(t,x,y) \in (0,\infty) \times F \times F$ and $b \in 2^{F^0}$. Now by \cite[Theorem 5.3.1(1)]{Kigami2001} we see that there exists $c_1(1) > 0$ such that
\begin{equation*}
\sup_{x,y \in F} p^N_t(x,y) \leq c_1(1) t^{-\frac{d_s}{2}}
\end{equation*}
for $t \in (0,1]$. Finally if $T > 1$ then the fact \cite[Proposition 5.1.2(1)]{Kigami2001} that $p^N$ is continuous in the compact set $[1,T] \times F \times F$ (and hence bounded in this set) implies that there exists $c_1(T) > 0$ such that the required result holds.
\item For $t > 0$ and $x,y \in F$ define $p^{b,x}_t(y) := p^b_t(x,y)$. By (1),
\begin{equation*}
\Vert p^{b,x}_t \Vert_\mu^2 = \int_F p^b_t(x,y)^2 \mu(dy) = p^b_{2t}(x,x) \leq 2^{-\frac{d_s}{2}}c_1(2T)t^{-\frac{d_s}{2}}
\end{equation*}
for all $t \in (0,T]$, $x \in F$ and $b \in 2^{F^0}$. Then we simply follow the method of \cite[Lemma 5.2]{Hambly1999}, using \cite[Lemma 1.3.3(i)]{Fukushima2011}.
\item For any $x,y \in F$ and $b \in 2^{F^0}$, $p^b_t(x,y)$ is differentiable in $t$ with derivative $\Delta_b p^{b,x}_t(y)$. This can be shown using Lemma \ref{thm:efuncgrow} in the same way as the proof of \cite[Proposition 5.1.2(4)]{Kigami2001}. The regularity of $p^b_t(x,y)$ and $\Delta_b p^{b,x}_t(y)$ is enough (again using Lemma \ref{thm:efuncgrow}) that for $(t,x) \in (0,\infty) \times F$ we can differentiate under the integral thus:
\begin{equation*}
\begin{split}
\frac{\partial}{\partial t}p^b_t(x,x) &= \frac{\partial}{\partial t}\int_F p^b_\frac{t}{2}(x,y)^2 \mu(dy) = 2\int_F p^{b,x}_\frac{t}{2}(y) \frac{\partial}{\partial t} p^{b,x}_\frac{t}{2}(y) \mu(dy) \\
&= 2\int_F p^{b,x}_\frac{t}{2}(y)\Delta_b p^{b,x}_\frac{t}{2}(y) \mu(dy) = -2 \Ecal \left(p^{b,x}_\frac{t}{2},p^{b,x}_\frac{t}{2} \right) \leq 0,
\end{split}
\end{equation*}
so $p^b_t(x,x)$ is decreasing in $t$ for any $x \in F$. Then \cite[Lemma 1.3.3(i)]{Fukushima2011} and the estimate for $\Vert p^{b,x}_t \Vert_\mu^2$ in the proof of (2) imply that there exists $c_3'(T) > 0$ such that
\begin{equation*}
-c_3'(T) t^{-1-\frac{d_s}{2}} \leq \frac{\partial}{\partial t}p^b_t(x,x) \leq 0
\end{equation*}
for all $(t,x) \in (0,T] \times F$. Therefore if $s,t \in (0,T]$ with $s \leq t$ then
\begin{equation*}
\begin{split}
\left| p^b_s(x,x) - p^b_t(x,x) \right| &\leq c_3'(T) \int_s^t z^{-1-\frac{d_s}{2}}dz\\
&= c_3'(T)\frac{2}{d_s} \left( s^{-\frac{d_s}{2}} - t^{-\frac{d_s}{2}} \right).
\end{split}
\end{equation*}
\end{enumerate}
\end{proof}
\begin{prop}[Stochastic continuity]\label{thm:stocts}
Fix $q \geq 2$ and $T > 0$. Then there exists $c_6 > 0$ such that the following holds: Let $v_0 \in \Scal_{q,T}$ and define
\begin{equation*}
\begin{split}
v_1(t,x) &= \int_0^t \int_F p^b_{t-s}(x,y) g(s,v_0(s,y)) \xi(s,y) \mu(dy)ds,\\
v_2(t,x) &= \int_0^t \int_F p^b_{t-s}(x,y) f(s,v_0(s,y)) \mu(dy)ds
\end{split}
\end{equation*}
for $(t,x) \in [0,T] \times F$. Then $v_1$ and $v_2$ are well-defined and for all $s,t \in [0,T]$ and $x,y \in F$,
\begin{equation*}
\begin{split}
\Ebb \left[ \left| v_1(t,x) - v_1(t,y) \right|^q \right] &\leq c_6 (1 + \Vert v_0 \Vert_{q,T}^q) R(x,y)^\frac{q}{4},\\
\Ebb \left[ \left| v_1(s,x) - v_1(t,x) \right|^q \right] &\leq c_6 (1 + \Vert v_0 \Vert_{q,T}^q) |s - t|^{\frac{q}{2}(1 - \frac{d_s}{2})},\\
\Ebb \left[ \left| v_2(t,x) - v_2(t,y) \right|^q \right] &\leq c_6 (1 + \Vert v_0 \Vert_{q,T}^q) R(x,y)^\frac{q}{4},\\
\Ebb \left[ \left| v_2(s,x) - v_2(t,x) \right|^q \right] &\leq c_6 (1 + \Vert v_0 \Vert_{q,T}^q) |s - t|^{\frac{q}{2}(1 - \frac{d_s}{2})}.\\
\end{split}
\end{equation*}
\end{prop}
\begin{proof}
We note that $v_1$ and $v_2$ are well-defined by the assumption on $v_0$ and the regularity of $p^b$ (Lemma \ref{thm:hkestim}(1)) and $f,g$. We now prove the spatial estimate for $v_1$. By the Burkholder-Davis-Gundy inequality, there exists a universal constant $C_q > 0$ such that
\begin{equation*}
\begin{split}
\Ebb &\left[ \left| v_1(t,x) - v_1(t,x') \right|^q \right] \\
&= \Ebb\left[ \left| \int_0^t \int_F \left( p^b_{t-s}(x,y) - p^b_{t-s}(x',y) \right) g(s,v_0(s,y)) \xi(s,y) \mu(dy)ds \right|^q \right]\\
&\leq C_q\Ebb\left[ \left| \int_0^t \int_F \left( p^b_{t-s}(x,y) - p^b_{t-s}(x',y) \right)^2 g(s,v_0(s,y))^2 \mu(dy)ds \right|^\frac{q}{2} \right]\\
&\leq C_q\left| \int_0^t \int_F \left( p^b_{t-s}(x,y) - p^b_{t-s}(x',y) \right)^2 \Ebb\left[ |g(s,v_0(s,y))|^q \right]^\frac{2}{q} \mu(dy)ds \right|^\frac{q}{2}\\
\end{split}
\end{equation*}
where in the last line we have used the Minkowski integral inequality. Now we have that $|g(s,v_0(s,y))|^q \leq(M(s) + C |v_0(s,y)|)^q \leq 2^{q-1}(M(s)^q + C^q|v_0(s,y)|^q)$ so using Lemma \ref{thm:hkestim}(2),
\begin{equation*}
\begin{split}
\Ebb &\left[ \left| v_1(t,x) - v_1(t,x') \right|^q \right] \\
&\leq 2^{q-1}C_q (\Vert M \Vert_{q,T}^q + C^q \Vert v_0 \Vert_{q,T}^q) \left| \int_0^t \int_F \left( p^b_{t-s}(x,y) - p^b_{t-s}(x',y) \right)^2 \mu(dy)ds \right|^\frac{q}{2}\\
&\leq 2^{q-1}C_q (\Vert M \Vert_{q,T}^q + C^q \Vert v_0 \Vert_{q,T}^q) \left| \int_0^t \left( p^b_{2(t-s)}(x,x) - 2 p^b_{2(t-s)}(x,x') + p^b_{2(t-s)}(x',x') \right) ds \right|^\frac{q}{2}\\
&\leq 2^{\frac{q}{4}(5-d_s)} c_2(2T)^\frac{q}{4} C_q (\Vert M \Vert_{q,T}^q + C^q \Vert v_0 \Vert_{q,T}^q) \left( \int_0^t (t-s)^{-\frac{1}{2} - \frac{d_s}{4}} ds \right)^\frac{q}{2} R(x,x')^\frac{q}{4}\\
&\leq 2^{\frac{q}{4}(5-d_s)} c_2(2T)^\frac{q}{4} C_q (\Vert M \Vert_{q,T}^q + C^q \Vert v_0 \Vert_{q,T}^q) \left( \int_0^T s^{-\frac{1}{2} - \frac{d_s}{4}} ds \right)^\frac{q}{2} R(x,x')^\frac{q}{4}\\
\end{split}
\end{equation*}
and the integral is finite since $d_s < 2$. Now we take care of the temporal estimate for $v_1$: Let $t,t' \in [0,T]$ with $t < t'$. Then as before,
\begin{equation*}
\begin{split}
\Ebb &\left[ \left| v_1(t,x) - v_1(t',x) \right|^q \right] \\
&\leq C_q\Ebb\left[ \left| \int_0^{t'} \int_F \left( p^b_{t-s}(x,y)\1bb_{\{s < t\}} - p^b_{t'-s}(x,y) \right)^2 g(s,v_0(s,y))^2 \mu(dy)ds \right|^\frac{q}{2} \right]\\
&\leq C_q\left| \int_0^{t'} \int_F \left( p^b_{t-s}(x,y)\1bb_{\{s < t\}} - p^b_{t'-s}(x,y) \right)^2 \Ebb\left[ |g(s,v_0(s,y))|^q \right]^\frac{2}{q} \mu(dy)ds \right|^\frac{q}{2}\\
&\leq 2^{q-1}C_q (\Vert M \Vert_{q,T}^q + C^q \Vert v_0 \Vert_{q,T}^q) \left| \int_0^{t'} \int_F \left( p^b_{t-s}(x,y)\1bb_{\{s < t\}} - p^b_{t'-s}(x,y) \right)^2 \mu(dy)ds \right|^\frac{q}{2}\\
&\leq 2^{q-1}C_q (\Vert M \Vert_{q,T}^q + C^q \Vert v_0 \Vert_{q,T}^q) \left| \int_0^{t'} \int_F \left( p^b_{t-s}(x,y)\1bb_{\{s < t\}} - p^b_{t'-s}(x,y) \right)^2 \mu(dy)ds \right|^\frac{q}{2}.
\end{split}
\end{equation*}
The integral in the final line above must be split into the sum of its parts $s < t$ and $s \geq t$. The first part is
\begin{equation*}
\begin{split}
\int_0^t \int_F &\left( p^b_{t-s}(x,y) - p^b_{t'-s}(x,y) \right)^2 \mu(dy)ds \\
&= \int_0^t \int_F \left( p^b_s(x,y) - p^b_{s + t' - t}(x,y) \right)^2 \mu(dy)ds\\
&= \int_0^t\left( p^b_{2s}(x,x) - 2 p^b_{2s + t' - t}(x,x) + p^b_{2(s + t' - t)}(x,x) \right) ds\\
&\leq \int_0^t\left( p^b_{2s}(x,x) - 2 p^b_{2s + t' - t}(x,x) + p^b_{2(s + t' - t)}(x,x) \right) ds\\
&\leq 2^{-\frac{d_s}{2}} c_3(2T)\int_0^t\left( s^{-\frac{d_s}{2}} - (s + t' - t)^{-\frac{d_s}{2}} \right) ds\\
&= 2^{-\frac{d_s}{2}} \left(1 - \frac{d_s}{2} \right)^{-1} c_3(2T) \left( t^{1-\frac{d_s}{2}} - (t')^{1-\frac{d_s}{2}} + (t' - t)^{1-\frac{d_s}{2}} \right)\\
&\leq 2^{-\frac{d_s}{2}} \left(1 - \frac{d_s}{2} \right)^{-1} c_3(2T) (t' - t)^{1-\frac{d_s}{2}}\\
\end{split}
\end{equation*}
where we have used Lemma \ref{thm:hkestim}(3). The second part is
\begin{equation*}
\begin{split}
\int_t^{t'} \int_F p^b_{t'-s}(x,y)^2 \mu(dy)ds &= \int_0^{t'-t} p^b_{2s}(x,x) ds\\
&\leq 2^{-\frac{d_s}{2}} c_1(2T) \int_0^{t'-t} s^{-\frac{d_s}{2}} ds\\
&\leq 2^{-\frac{d_s}{2}} c_1(2T) \left( 1 - \frac{d_s}{2} \right)^{-1} (t' - t)^{1-\frac{d_s}{2}}\\
\end{split}
\end{equation*}
where we have used Lemma \ref{thm:hkestim}(1). Together these give us the temporal estimate for $v_1$.

The respective estimates for $v_2$ can be found similarly, though they are generally easier as there is no noise to deal with. In particular, we use Jensen's instead of the Burkholder-Davis-Gundy inequality.
\end{proof}
\begin{cor}\label{thm:pred}
Fix $q \geq 2$ and $T > 0$. Let $v_0 \in \Scal_{q,T}$ and define $v_1, v_2$ as in Proposition \ref{thm:stocts}. Then $v_1,v_2 \in \Scal_{q,T}$.
\end{cor}
\begin{proof}
The estimates found in Proposition \ref{thm:stocts} show us that $v_1$ and $v_2$ are uniformly stochastically continuous on $[0,T] \times F$. For each fixed $t \in [0,T]$, $v_1(t,\cdot)$ and $v_2(t,\cdot)$ are jointly $\Omega \times F$-measurable, which we can check by approximating the continuous map $(s,x,y) \mapsto p^b_{t-s}(x,y)$ pointwise from below by a sequence of finite positive linear combinations of rectangles $(s,x,y) \mapsto \1bb_{A_1}(s))\1bb_{A_2}(x)\1bb_{A_3}(y)$ such that $A_1 \times A_2 \times A_3 \subseteq  [0,t) \times F \times F$.

Since $v_1$ and $v_2$ are also both evidently adapted, a standard argument (see for example \cite[Proposition 3.7(ii)]{DaPrato1992}) shows that they have predictable versions. As a rough sketch, for $n \geq 1$ let
\begin{equation*}
v_1^n(t,x) = \sum_{i=0}^{2^n-1} v_1\left(\frac{i}{2^n}T,x\right)\1bb_{(\frac{i}{2^n}T,\frac{i+1}{2^n}T]}(t)
\end{equation*}
for $(t,x) \in [0,T] \times F$. Then each $v_1^n$ is predictable and we can show that $\Vert v_1^n - v_1 \Vert_{q,T} \to 0$ as $n \to \infty$.

It remains to show that $\Vert v_i \Vert_{q,T} < \infty$ for $i = 1,2$. This is easy to see by setting $s=0$ in Proposition \ref{thm:stocts}, since $v_i(0,\cdot) = 0$.
\end{proof}
To prove existence and uniqueness we follow closely the methods of \cite[Theorem 3.2]{Walsh1986} and \cite[Theorem 5.5]{Khoshnevisan2014}.
\begin{thm}[Existence and uniqueness]\label{thm:exist}
Assume Hypothesis \ref{walshhyp}. Then the SPDE \eqref{WSPDE} has a unique mild solution $u$ in $\Scal_{q,T}$.
\end{thm}
\begin{proof}
\textit{Uniqueness}: We may assume without loss of generality that $q = 2$. Suppose that $u_1, u_2 \in \Scal_{2,T}$ are both mild solutions to \eqref{WSPDE}. Let $v = u_1 - u_2 \in \Scal_{2,T}$. For $(t,x) \in [0,T] \times F$ let $G(t) = \sup_{x \in F} \Ebb\left[ v(t,x)^2 \right]$. Note that $G$ is non-negative and bounded in $[0,T]$ since $v \in \Scal_{2,T}$. Then for $(t,x) \in [0,T] \times F$ we have by the Burkholder-Davis-Gundy inequality that
\begin{equation*}
\begin{split}
\Ebb\left[ v(t,x)^2 \right] \leq & 2T \Ebb \left[\int_0^t \int_F p^b_{t-s}(x,y)^2 (f(s,u_1(s,y)) - f(s,u_2(s,y)))^2 \mu(dy)ds \right]\\
&+ 2 \Ebb\left[ \int_0^t \int_F p^b_{t-s}(x,y)^2 (g(s,u_1(s,y)) - g(s,u_2(s,y)))^2 \mu(dy)ds \right]
\end{split}
\end{equation*}
so that
\begin{equation*}
\begin{split}
\Ebb\left[ v(t,x)^2 \right] &\leq 2C^2(T+1) \Ebb\left[\int_0^t \int_F p^b_{t-s}(x,y)^2 (u_1(s,y) - u_2(s,y))^2 \mu(dy)ds \right]\\
&\leq 2C^2(T+1) \int_0^t G(s) \int_F p^b_{t-s}(x,y)^2 \mu(dy)ds\\
&= 2C^2(T+1) \int_0^t G(s) p^b_{2(t-s)}(x,x) ds\\
\end{split}
\end{equation*}
and thus by Lemma \ref{thm:hkestim}(1), for all $t \in [0,T]$,
\begin{equation*}
G(t) \leq 2^{1-\frac{d_s}{2}}C^2(T+1)c_1(2T) \int_0^t G(s) (t-s)^{-\frac{d_s}{2}} ds.
\end{equation*}
Now $d_s < 2$, so by \cite[Lemma 3.3]{Walsh1986} we have that $G(t) = 0$ for all $t \in [0,T]$. This concludes the proof of uniqueness.

\textit{Existence}: We find a solution in $\Scal_{q,T}$ by Picard iteration. Let $u_1 = 0 \in \Scal_{q,T}$, then for $n \geq 1$ we would like to let $u_{n+1} = \{ u_{n+1}(t,x): (t,x) \in [0,T] \times F \}$ be the predictable version (using Proposition \ref{thm:stocts}) of the process defined by
\begin{equation}\label{induct}
\begin{split}
u_{n+1}(t,x) = &\int_F p^b_t(x,y) u_0(y) \mu(dy) + \int_0^t \int_F p^b_{t-s}(x,y) f(s,u_n(s,y)) \mu(dy)ds\\
&+ \int_0^t \int_F p^b_{t-s}(x,y) g(s,u_n(s,y)) \xi(s,y) \mu(dy)ds.
\end{split}
\end{equation}
We need to show that this sequence is well-defined. By Proposition \ref{thm:stocts} it will be enough to prove that if $u_n \in \Scal_{q,T}$ for some $n$ and $u_{n+1}$ is defined as in \eqref{induct}, then $u_{n+1} \in \Scal_{q,T}$. We do this by induction on $n$. Assume that $u_n \in \Scal_{q,T}$ and define $u_{n+1}$ as in \eqref{induct}. By Proposition \ref{thm:stocts} and Corollary \ref{thm:pred}, the second and third term on the right-hand side of \eqref{induct} are well-defined elements of $\Scal_{q,T}$, which leaves us to deal with the first term $\int_F p^b_t(x,y) u_0(y) \mu(dy)$. This is trivially predictable as it is $\Fcal_0$-measurable and, by dominated convergence using Lemma \ref{thm:hkestim}, it is almost surely continuous in $(0,T] \times F$. Finally using Minkowski's integral inequality,
\begin{equation}\label{eqn:heatMinkow}
\begin{split}
\Ebb \left[ \left|\int_F p^b_t(x,y) u_0(y) \mu(dy) \right|^q \right] &\leq \left(\int_F p^b_t(x,y) \Ebb \left[ |u_0(y)|^q \right]^\frac{1}{q} \mu(dy) \right)^q\\
&\leq \Vert u_0 \Vert_q^q \left|\int_F p^b_t(x,y) \mu(dy) \right|^q\\
&\leq \Vert u_0 \Vert_q^q
\end{split}
\end{equation}
for all $(t,x) \in [0,T] \times F$, showing that the first term of the right-hand side of \eqref{induct} is indeed an element of $\Scal_{q,T}$. So $u_{n+1} \in \Scal_{q,T}$, and in conclusion $(u_n)_{n=1}^\infty$ is a well-defined sequence of elements of $\Scal_{q,T}$.

We now show that the sequence $(u_n)_{n=1}^\infty$ is Cauchy. For $n \geq 1$ let $v_n = u_{n+1} - u_n \in \Scal_{q,T}$. Then by Jensen's and Burkholder-Davis-Gundy inequalities, there exists a constant $C_q > 0$ such that for all $(t,x) \in [0,T] \times F$ we have that
\begin{equation*}
\begin{split}
&\Ebb\left[ |v_{n+1}(t,x)|^q \right] \\
&\leq  2^{q-1} T^\frac{q}{2}\Ebb\left[ \left( \int_0^t \int_F p^b_{t-s}(x,y)^2 \left( f(s,u_{n+1}(s,y)) - f(s,u_n(s,y)) \right)^2 \mu(dy)ds \right)^\frac{q}{2} \right]\\
&\phantom{=} + 2^{q-1} C_q\Ebb\left[ \left( \int_0^t \int_F p^b_{t-s}(x,y)^2 \left( g(s,u_{n+1}(s,y)) - g(s,u_n(s,y)) \right)^2 \mu(dy)ds \right)^\frac{q}{2} \right].
\end{split}
\end{equation*}
Using the Lipschitz property of $f$ and $g$ and then Minkowski's integral inequality this implies that
\begin{equation*}
\begin{split}
\Ebb\left[ |v_{n+1}(t,x)|^q \right] &\leq  2^{q-1} \left( T^\frac{q}{2} + C_q \right) C^q \Ebb\left[ \left( \int_0^t \int_F p^b_{t-s}(x,y)^2 v_n(s,y)^2 \mu(dy)ds \right)^\frac{q}{2} \right]\\
&\leq  2^{q-1} \left( T^\frac{q}{2} + C_q \right) C^q \left( \int_0^t \int_F p^b_{t-s}(x,y)^2 \Ebb\left[ |v_n(s,y)|^q \right]^\frac{2}{q} \mu(dy)ds \right)^\frac{q}{2}\\
\end{split}
\end{equation*}
Let $H_n(t) = \sup_{x \in F} \Ebb\left[ |v_n(t,x)|^q \right]^\frac{2}{q}$ for $n \geq 1$ and $t \in [0,T]$. Since $v_n \in \Scal_{q,T}$, each $H_n$ must be bounded in $[0,T]$. Then the above equation implies that there is a constant $C > 0$ such that for all $n$ and $(t,x)$,
\begin{equation*}
\begin{split}
\Ebb\left[ |v_{n+1}(t,x)|^q \right]^\frac{2}{q} &\leq C \int_0^t H_n(s) \int_F p^b_{t-s}(x,y)^2 \mu(dy)ds\\
&= C \int_0^t H_n(s) p^b_{2(t-s)}(x,x) ds\\
&\leq 2^{-\frac{d_s}{2}} c_1(2T) C \int_0^t H_n(s) (t-s)^{-\frac{d_s}{2}} ds,\\
\end{split}
\end{equation*}
where we have used Lemma \ref{thm:hkestim}(1). This implies that
\begin{equation*}
H_{n+1}(t) \leq 2^{-\frac{d_s}{2}} c_1(2T) C \int_0^t H_n(s) (t-s)^{-\frac{d_s}{2}} ds
\end{equation*}
for $t \in [0,T]$. By \cite[Lemma 3.3]{Walsh1986} and the fact that $d_s < 2$, there then exists a constant $C_0 > 0$ and an integer $k \geq 1$ such that for each $n,m \geq 1$ and $t \in [0,T]$ we have that
\begin{equation*}
H_{n + mk}(t) \leq \frac{C_0^m}{(m-1)!} \int_0^t H_n(s) (t-s) ds.
\end{equation*}
Thus for each $n \geq 1$, $\sum_{m=0}^\infty H_{n + mk}^\frac{1}{2}$ converges uniformly in $[0,T]$. This implies that $\sum_{n=1}^\infty H_n^\frac{1}{2}$ also converges uniformly in $[0,T]$ and so the sequence $(u_n)_{n=1}^\infty$ is Cauchy in $\Scal_{q,T}$. Let $u \in \Scal_{q,T}$ be the limit of this sequence. Now for each $(t,x) \in [0,T] \times F$ we take the limit $n \to \infty$ in $L^q(\Omega)$ on both sides of \eqref{induct}. The left-hand side tends to $u(t,x)$, whereas the right-hand side tends to
\begin{equation*}
\begin{split}
&\int_F p^b_t(x,y) u_0(y) \mu(dy) + \int_0^t \int_F p^b_{t-s}(x,y) f(s,u(s,y)) \mu(dy)ds\\
&+ \int_0^t \int_F p^b_{t-s}(x,y) g(s,u(s,y)) \xi(s,y) \mu(dy)ds
\end{split}
\end{equation*}
by a calculation that by now is routine. Therefore $u$ is a mild solution to the SPDE \eqref{WSPDE} defined on $[0,T] \times F$.
\end{proof}

\subsection{Continuous random field version}

Assuming Hypothesis \ref{walshhyp}, let $u \in \Scal_{q,T}$ be the mild solution to \eqref{WSPDE} as obtained in Theorem \ref{thm:exist}. Using Corollary \ref{thm:pred}, let $u^{\sto} = \{ u^{\sto}(t,x) : (t,x) \in [0,T] \times F \}$ be the element of $\Scal_{q,T}$ satisfying
\begin{equation*}
\begin{split}
u^{\sto}(t,x) &= \int_0^t \int_F p^b_{t-s}(x,y) f(s,u(s,y)) \mu(dy)ds\\
&\phantom{=} + \int_0^t \int_F p^b_{t-s}(x,y) g(s,u(s,y)) \xi(s,y) \mu(dy)ds\\
\end{split}
\end{equation*}
almost surely for each $(t,x) \in [0,T] \times F$. So $u^{\sto}$ is the ``stochastic part'' of $u$ and satisfies
\begin{equation*}
u^{\sto}(t,x) = u(t,x) - \int_F p^b_t(x,y) u_0(y) \mu(dy)
\end{equation*}
almost surely for each $(t,x) \in [0,T] \times F$. In this section we prove that $u^{\sto}$ has a version which is a continuous random field satisfying relatively weak continuity properties, and then we use Theorem \ref{mainthm} to bootstrap these into stronger H\"older continuity properties.
\begin{lem}\label{thm:heqcts1}
Let $b \in 2^{F^0}$. Let $h \in \Hcal$. Consider the map from $(0,\infty) \times F$ to $\Rbb$ given by
\begin{equation*}
(t,x) \mapsto \int_F p^b_t(x,y)h(y)\mu(dy).
\end{equation*}
Then for every $0 < T_1 < T_2$, on $[T_1,T_2] \times F$ this map is $\frac{1}{2}$-H\"older continuous with respect to $R_\infty$.
Moreover, if $\sup_{x \in F} |h(x)| < \infty$ then 
\begin{equation*}
\sup_{(t,x) \in (0,\infty) \times F} \left| \int_F p^b_t(x,y)h(y)\mu(dy) \right| \leq \sup_{x \in F} |h(x)|.
\end{equation*}
\end{lem}
\begin{proof}
Fix $0 < T_1 < T_2$. It is enough to prove that the map is uniformly H\"older continuous in each argument. By Lemma \ref{thm:hkestim}, for any $t \in [T_1,T_2]$ and $x,y \in F$,
\begin{equation*}
\left| \int_F (p^b_t(x,z) - p^b_t(y,z))h(z)\mu(dz) \right| \leq c_2(T_2)^\frac{1}{2} T_1^{-\frac{1}{2} - \frac{d_s}{4}} R(x,y)^\frac{1}{2} \int_F |h(z)|\mu(dz).
\end{equation*}
Similarly for any $s,t \in [T_1,T_2]$ with $s < t$ and $x \in F$, the fact that $t' \mapsto p^b_{t'}(x,x)$ is decreasing (evident by definition of $p^b$) implies that
\begin{equation*}
\begin{split}
&\left( \int_F (p^b_s(x,z) - p^b_t(x,z))h(z)\mu(dz) \right)^2\\
&\leq \int_F (p^b_s(x,z) - p^b_t(x,z))^2\mu(dz) \int_F h(z)^2\mu(dz)\\
&\leq \left( p^b_{2s}(x,x) - 2p^b_{s+t}(x,x) + p^b_{2t}(x,x) \right) \int_F h(z)^2\mu(dz)\\
&\leq \left( p^b_{2s}(x,x) - p^b_{s+t}(x,x) \right) \int_F h(z)^2\mu(dz)\\
&\leq C_3(2T_2)\left( (2s)^{-\frac{d_s}{2}} - (s+t)^{-\frac{d_s}{2}} \right) \int_F h(z)^2\mu(dz)\\
&\leq C_3(2T_2)\left( s^{-\frac{d_s}{2}} - t^{-\frac{d_s}{2}} \right) \int_F h(z)^2\mu(dz).
\end{split}
\end{equation*}
Now $t \mapsto t^{-\frac{d_s}{2}}$ is Lipschitz in $[T_1,T_2]$ so we have the required result.

For the last claim, if $\sup_{x \in F} |h(x)| = C$ then
\begin{equation*}
\left| \int_F p^b_t(x,y)h(y)\mu(dy) \right| \leq C\int_F p^b_t(x,y)\mu(dy) \leq C
\end{equation*}
for all $(t,x) \in (0,\infty) \times F$.
\end{proof}
Our first result on $u^{\sto}$ uses directly the stochastic continuity results of the previous section.
\begin{prop}\label{thm:stoctscty}
Assume Hypothesis \ref{walshhyp} with $q > 2(d_H + 1)^2$. Then there exists $\gamma \in (0,1]$ such that $u^{\sto}$ has a version $\tilde{u}^{\sto}$ which is predictable and $\gamma$-H\"older continuous on $[0,T] \times F$ with respect to $R_\infty$ almost surely.
\end{prop}
\begin{proof}
Given the stochastic continuity estimates in Proposition \ref{thm:stocts} and the fact that $d_H \geq 1$ (from \cite[Remark 2.6(2)]{Hambly2016}), the condition $q > 2(d_H + 1)^2$ is precisely what is needed for us to be able to use \cite[Theorem 3.17]{Hambly2016} to construct an almost surely H\"older continuous version of $u^{\sto}$ on $[0,T] \times F$ for any $T > 0$. It is adapted, so its continuity immediately implies that it is predictable.
\end{proof}
Let $C(F)$ be the Banach space of continuous functions from $F$ to $\Rbb$ equipped with uniform norm $\Vert \cdot \Vert_\infty$. We use the above two results to construct a process in this space which is ``almost'' a version of $u$.
\begin{cor}\label{thm:leftctsproc}
Assume Hypothesis \ref{walshhyp} with $q > 2(d_H + 1)^2$. Then there exists a predictable $C(F)$-valued process $(\tilde{u}(t,\cdot), t \in [0,T])$ such that $u(0,\cdot) = 0$, and $\tilde{u}(t,x) = u(t,x)$ almost surely for $(t,x) \in (0,T] \times F$. Moreover, $t \mapsto \tilde{u}(t,\cdot)$ is continuous from $(0,T]$ to $C(F)$ almost surely.
\end{cor}
\begin{proof}
Define $\tilde{u}(t,x)$ as follows: for $t = 0$ let $\tilde{u}(0,x) = 0$ for all $x \in F$. For all $(t,x) \in (0,T] \times F$ let
\begin{equation*}
\tilde{u}(t,x) = \int_F p^b_t(x,y) u_0(y) \mu(dy) + \tilde{u}^{\sto}(t,x),
\end{equation*}
Where $\tilde{u}^{\sto}$ is as defined in Proposition \ref{thm:stoctscty}. We therefore obviously have $\tilde{u}(t,x) = u(t,x)$ almost surely for $(t,x) \in (0,T] \times F$. Hypothesis \ref{walshhyp} implies that $u_0 \in \Hcal$ almost surely, so Lemma \ref{thm:heqcts1} and Proposition \ref{thm:stoctscty} imply that $\tilde{u}$ is almost surely H\"older continuous in $[T_1,T] \times F$ with respect to $R_\infty$ for any $T_1 \in (0,T)$. Thus $t \mapsto \tilde{u}(t,\cdot)$ is well-defined as a $C(F)$-valued process and is almost surely continuous from $(0,T]$ to $C(F)$. It is therefore almost surely left-continuous on $[0,T]$. This combined with its adaptedness implies that it is predictable.
\end{proof}
\begin{lem}\label{thm:betasigma}
Assume Hypothesis \ref{walshhyp} with $q > 2(d_H + 1)^2$. Let $\tilde{u}$ be as given in Corollary \ref{thm:leftctsproc}. Let $\beta = \{ \beta(t) : t \in [0,T] \}$ be the $\Hcal$-valued process given by $\beta(t)(x) = f(t,\tilde{u}(t,x))$ for $x \in F$, and let $\sigma = \{ \sigma(t) : t \in [0,T] \}$ be the $\Lcal(\Hcal)$-valued process given by $\sigma(t) = \Mcal_{( x \mapsto g(t,\tilde{u}(t,x)))}$. Then $\beta$ and $\sigma$ are predictable processes on their respective spaces $\Hcal$ and $\Lcal(\Hcal)$.
\end{lem}
\begin{proof}
Recall that $t \mapsto \tilde{u}$ is a predictable $C(F)$-valued process. To prove this result we show that the processes $\beta$ and $\sigma$ are compositions of suitably measurable functions of $\tilde{u}$.

To show that $\beta$ is predictable, let $\bar{f}:\Omega \times [0,\infty) \times \Hcal \to \Hcal$ be given by $\bar{f}(\omega,t,h) = (x \mapsto f(\omega,t,h(x)))$. For fixed $(\omega,t) \in \Omega \times [0,\infty)$ this is continuous in $h$ since $f$ is Lipschitz in $x$. For fixed $h \in \Hcal$ this is predictable; indeed for any closed $\Hcal$-ball $\bar{B}(h',\epsilon) \in \Bcal(\Hcal)$,
\begin{equation*}
\left\{(\omega,t):f(\omega,t,h(\cdot)) \in \bar{B}(h',\epsilon) \right\} = \left\{(\omega,t): \int_F(f(\omega,t,h(x)) - h'(x))^2\mu(dx) \leq \epsilon \right\}
\end{equation*}
which is predictable because $(\omega,t,x) \mapsto f(\omega,t,h(x))$ is predictable. Therefore by \cite[Lemma 4.51]{Aliprantis2006}, $\bar{f}$ is $\Pcal_T \otimes \Bcal(\Hcal)$-measurable. So
\begin{equation*}
\beta(t) = \bar{f}(t,\iota(\tilde{u}(t,\cdot))),
\end{equation*}
where $\iota: C(F) \hookrightarrow \Hcal$ is the continuous inclusion map. Thus $\beta$ is a predictable $\Hcal$-valued process.

Now we deal with $\sigma$, which is slightly different. let $\bar{g}:\Omega \times [0,\infty) \times C(F) \to C(F)$ be given by $\bar{g}(\omega,t,\eta) = (x \mapsto g(\omega,t,\eta(x)))$. For fixed $(\omega,t) \in \Omega \times [0,\infty)$ this is continuous in $\eta$ since $g$ is Lipschitz in $x$. For fixed $\eta \in C(F)$ this is predictable; indeed, let $\{ x_i \}_i$ be a countable dense subset of $F$. Then for any closed $C(F)$-ball $\bar{B}(\eta',\epsilon) \in \Bcal(C(F))$,
\begin{equation*}
\left\{(\omega,t):g(\omega,t,\eta(\cdot)) \in \bar{B}(\eta',\epsilon) \right\} = \bigcap_i\left\{(\omega,t): |g(\omega,t,\eta(x_i)) - \eta'(x_i)| \leq \epsilon \right\},
\end{equation*}
which is predictable because $(\omega,t,x) \mapsto g(\omega,t,h(x))$ is predictable. Therefore by \cite[Lemma 4.51]{Aliprantis2006}, $\bar{g}$ is $\Pcal_T \otimes \Bcal(C(F))$-measurable. Now consider the function $\eta \mapsto \Mcal_\eta$ from $C(F)$ to $\Lcal(\Hcal)$. This is a continuous linear map since $\Vert \Mcal_\eta \Vert = \Vert \eta \Vert_\infty$ for all $\eta \in C(F)$. So
\begin{equation*}
\sigma(t) = \Mcal_{\bar{g}(t,\hat{u}(t,\cdot))}
\end{equation*}
which is now clearly a predicable $\Lcal(\Hcal)$-valued process.
\end{proof}
We may now prove our H\"older regularity result for $u^{\sto}$:
\begin{thm}
Assume Hypothesis \ref{walshhyp} with $q > 2(d_H + 1)^2$. Then $u^{\sto}$ has a version $\tilde{u}^{\sto}$ with the following H\"older continuity properties:
\begin{enumerate}
\item $\tilde{u}^{\sto}$ is almost surely essentially $\left( \frac{1}{2} (d_H + 1)^{-1} - \frac{d_H + 1}{q} \right)$-H\"{o}lder continuous in $[0,T] \times F$ with respect to $R_\infty$,
\item For each $t \in [0,T]$, $\tilde{u}^{\sto}(t,\cdot)$ is almost surely essentially $\left( \frac{1}{2} - \frac{d_H}{q} \right)$-H\"{o}lder continuous in $F$ with respect to $R$,
\item For each $x \in F$, $\tilde{u}^{\sto}(\cdot,x)$ is almost surely essentially $\left( \frac{1}{2} (d_H + 1)^{-1} - \frac{1}{q} \right)$-H\"{o}lder continuous in $[0,T]$.
\end{enumerate}
\end{thm}
\begin{proof}
We take $\tilde{u}^{\sto}$ to be the almost surely continuous version of $u^{\sto}$ as defined in Proposition \ref{thm:stoctscty}. Recall $\tilde{u}$ as defined in Corollary \ref{thm:leftctsproc}; we see that
\begin{equation*}
\begin{split}
\tilde{u}^{\sto}(t,x) &= \int_0^t \int_F p^b_{t-s}(x,y) f(s,\tilde{u}(s,y)) \mu(dy)ds\\
&\phantom{=} + \int_0^t \int_F p^b_{t-s}(x,y) g(s,\tilde{u}(s,y)) \xi(s,y) \mu(dy)ds
\end{split}
\end{equation*}
almost surely for each $(t,x) \in [0,T] \times F$. This holds since the fact that $\tilde{u}$ does not agree almost surely with $u$ for $t=0$ has no effect on integrals or stochastic integrals. Define an $\Hcal$-valued predictable process $U = (U(t): t \in [0,T])$ by $U(t) = \tilde{u}^{\sto}(t,\cdot)$, and let $\beta$ and $\sigma$ be defined as in Lemma \ref{thm:betasigma}. Using the connections between the space-time white noise $\xi$ and the cylindrical Wiener process $W$ we can rewrite the definition of $\tilde{u}^{\sto}$ on the level of $\Hcal$:
\begin{equation*}
U(t) = \int_0^t S^b_{t-s} \beta(s) ds + \int_0^t S^b_{t-s} \sigma(s) dW(s)
\end{equation*}
almost surely for each $t \in [0,T]$. This equation can be verified by taking $\langle \cdot,h_i \rangle_\mu$ on both sides where $h_i$ are the elements of an orthonormal basis $\{ h_i \}_i$ of $\Hcal$ and then using (stochastic) Fubini's theorem \cite[Theorem 2.6]{Walsh1986}. Now we see that
\begin{equation*}
\begin{split}
\Ebb\left[ \left( \int_0^T \Vert \beta(s) \Vert_\mu^2 ds \right)^\frac{q}{2} \right] &= \Ebb\left[ \left( \int_0^T \int_F f(s,\tilde{u}(s,x))^2 \mu(dx) ds \right)^\frac{q}{2} \right]\\
&\leq T^{\frac{q}{2} - 1} \Ebb\left[ \int_0^T \int_F |M(s) + C\tilde{u}(s,x)|^q \mu(dx) ds \right]\\
&\leq T^\frac{q}{2} \sup_{(s,x) \in [0,T] \times F} \Ebb\left[ |M(s) + C\tilde{u}(s,x)|^q \right]\\
&< \infty.
\end{split}
\end{equation*}
In addition, $(\omega,s,x) \mapsto g(\omega,s,\tilde{u}(\omega,s,x))$ is evidently jointly measurable and
\begin{equation*}
\begin{split}
\sup_{(s,x) \in [0,T] \times F} \Ebb\left[ | g(s,\tilde{u}(s,x)) |^q \right] &\leq \sup_{(s,x) \in [0,T] \times F} \Ebb\left[ | M(s) + C\tilde{u}(s,x) |^q \right]\\
&< \infty.
\end{split}
\end{equation*}
Therefore Hypothesis \ref{hyp2} holds for $p = \frac{q}{2}$. Thus by applying Theorem \ref{mainthm} there exists a function $\hat{u}: \Omega \times [0,T] \times F$ such that $\hat{u}(t,x)$ is a random variable for each $(t,x) \in [0,T] \times F$, $\hat{u}$ is almost surely continuous in $[0,T] \times F$ with the required H\"older exponents given in the statement of the present proposition, and $\hat{u}(t,\cdot) = U(t) = \tilde{u}^{\sto}(t,\cdot)$ (in the sense of elements of $\Hcal$) almost surely for each $t \in [0,T]$. We now show that $\hat{u}(t,x) = \tilde{u}^{\sto}(t,x)$ almost surely for each $(t,x) \in [0,T] \times F$. But this is clear: $\hat{u}(t,\cdot) = \tilde{u}^{\sto}(t,\cdot)$ in $\Hcal$ almost surely, and both are almost surely continuous in $F$, so we must have that $\hat{u}(t,x) = \tilde{u}^{\sto}(t,x)$ almost surely for all $x \in F$. Now since $[0,T] \times F$ is separable, any two almost surely continuous versions of $u^{\sto}$ must be indistinguishable on their domains of definition. Therefore $\hat{u} = \tilde{u}^{\sto}$ in $[0,T] \times F$ almost surely and the proof is complete.
\end{proof}

\section{Bounds on global solutions and weak intermittency}\label{sec:intermitt}

We seek upper and lower moment bounds on the global solutions to Walsh SPDEs on fractals. A corollary of these results is that a certain class of 
these SPDEs, which includes a version of the parabolic Anderson model, exhibits \textit{intermittency}, see \cite{Khoshnevisan2014} and further references.
\begin{lem}\label{lem:heatbounds}
There exist $c_7,c_8 > 0$ such that
\begin{equation*}
c_7 \left( 1 + t^{-\frac{d_s}{2}} \right) \leq p^N_t(x,x) \leq c_8 \left( 1 + t^{-\frac{d_s}{2}} \right)
\end{equation*}
for all $(t,x) \in (0,\infty) \times F$.
\end{lem}
\begin{proof}
Recall that
\begin{equation*}
p^N_t(x,y) = \sum_{k=1}^\infty e^{-\lambda^N_k t} \phi^N_k(x)\phi^N_k(y).
\end{equation*}
Therefore by \cite[Theorem 4.5.4]{Kigami2001} and the fact that $\lambda^N_1 = 0 < \lambda^N_2$ and $\phi^N_1 \equiv 1$, the map $x \mapsto p^N_t(x,x)$ must converge to $1$ uniformly as $t \to \infty$. Then \cite[Theorem 5.3.1]{Kigami2001} implies the result.
\end{proof}
In particular, for all $b \in 2^{F^0}$, we have that $p^b_t(x,x) \leq c_8 \left( 1 + t^{-\frac{d_s}{2}} \right)$ for all $(t,x) \in [0,\infty) \times F$.
\begin{defn}
Let $\kappa: [0,\infty) \to [0,\infty)$ be the function such that $\kappa(0) = 0$ and for $\alpha > 0$,
\begin{equation*}
\kappa(\alpha) \int_0^\infty e^{-\alpha t} \left( 1 + t^{-\frac{d_s}{2}} \right) dt = 1.
\end{equation*}
\end{defn}
Indeed, direct computation yields
\begin{equation}\label{eqn:kappaform}
\kappa(\alpha) = \frac{\alpha}{1 + \alpha^\frac{d_s}{2}\Gamma(1 - d_s/2)}
\end{equation}
where $\Gamma$ is the Gamma function. The following properties are easily verified:
\begin{enumerate}
\item $\kappa$ is continuous in $[0,\infty)$ and continuously differentiable in $(0,\infty)$ with positive first derivative.
\item $\kappa(\alpha) = \Theta(\alpha)$ as $\alpha \to 0$ and $\kappa(\alpha) = \Theta(\alpha^{1-\frac{d_s}{2}})$ as $\alpha \to \infty$.
\end{enumerate}
The above two properties imply that $\kappa$ is a strictly increasing bijection on $[0,\infty)$. Its first derivative is given by
\begin{equation*}
\kappa'(\alpha) = \frac{1 + \alpha^\frac{d_s}{2}\Gamma(2 - d_s / 2)}{(1 + \alpha^\frac{d_s}{2}\Gamma(1 - d_s / 2))^2}.
\end{equation*}
\begin{cor}\label{cor:heatintbd}
For $\alpha > 0$ and $b \in 2^{F^0}$,
\begin{equation*}
\kappa(\alpha) \int_0^\infty e^{-\alpha t} p^b_t(x,x) dt \leq c_8.
\end{equation*}
Additionally in the case $b = N$,
\begin{equation*}
c_7 \leq \kappa(\alpha) \int_0^\infty e^{-\alpha t} p^N_t(x,x) dt.
\end{equation*}
\end{cor}
\begin{proof}
Directly from Lemma \ref{lem:heatbounds} and subsequent discussion.
\end{proof}

\subsection{Upper moment bound}

Let $b \in 2^{F^0}$ and consider the following SPDE on $F$ for time $t \in [0,\infty)$:
\begin{equation}\label{eqn:compSPDE1}
\begin{split}
\frac{\partial u}{\partial t}(t,x) &= \Delta_b u(t,x) + f(t,u(t,x)) + g(t,u(t,x))\xi(t,x),\\
u(0,x) &= u_0(x).
\end{split}
\end{equation}
\begin{hyp}\label{hyp:easywalshhyp}
We make the following assumptions:
\begin{enumerate}
\item $u_0: F \to \Rbb$ is measurable and bounded.
\item $f,g: \Omega \times [0,\infty) \times \Rbb \to \Rbb$ are functions which are measurable from $\Pcal \otimes \Bcal(\Rbb)$ to $\Bcal(\Rbb)$. There exists a constant $C > 0$ such that for all $(\omega,t) \in \Omega \times [0,\infty)$ and all $x,y \in \Rbb$,
\begin{equation*}
\begin{split}
|f(\omega,t,x) - f(\omega,t,y)| + |g(\omega,t,x) - g(\omega,t,y)| &\leq C|x - y|,\\
|f(\omega,t,x)| + |g(\omega,t,x)| &\leq C(1+|x|).
\end{split}
\end{equation*}
\end{enumerate}
\end{hyp}
We may now find an upper bound for the moments of the solution to \eqref{eqn:compSPDE1}. Compare \cite[Theorem 5.5]{Khoshnevisan2014}.
\begin{thm}[General moment upper bound]\label{thm:uppermmtbd}
Assume Hypothesis \ref{hyp:easywalshhyp}. Let $u$ be the solution to \eqref{eqn:compSPDE1}. Then there exist $c_9,c_{10} > 0$ such that for all $p \geq 1$ and all $(t,x) \in [0,\infty) \times F$,
\begin{equation*}
\Ebb \left[\left| u(t,x) \right|^p \right]^\frac{1}{p} \leq c_9 \exp\left(c_{10} p^{1+d_H} t\right).
\end{equation*}
\end{thm}
\begin{proof}
Let $p \geq 2$ and $\alpha > 0$. Let $(t,x) \in [0,\infty) \times F$. We see that
\begin{equation*}
\begin{split}
e^{-\alpha t}\Ebb\left[ |u(t,x)|^p \right]^\frac{1}{p} \leq
& e^{-\alpha t} \left| \int_F p^b_t(x,y) u_0(y) \mu(dy)\right|\\
&+ \Ebb\left[\left|\int_0^t \int_F e^{-\alpha t}p^b_{t-s}(x,y) f(s,u(s,y)) \mu(dy)ds\right|^p\right]^\frac{1}{p}\\
&+ \Ebb\left[\left|\int_0^t \int_F e^{-\alpha t} p^b_{t-s}(x,y) g(s,u(s,y)) \xi(s,y) \mu(dy)ds\right|^p\right]^\frac{1}{p},
\end{split}
\end{equation*}
which implies
\begin{equation*}
\begin{split}
e^{-\alpha t}\Ebb\left[ |u(t,x)|^p \right]^\frac{1}{p} \leq
& \sup_{x \in F}|u_0(x)|\\
&+ \Ebb\left[\left|\int_0^t \int_F e^{-\alpha t}p^b_{t-s}(x,y) f(s,u(s,y)) \mu(dy)ds\right|^p\right]^\frac{1}{p}\\
&+ 2 \sqrt{p} \Ebb\left[\left|\int_0^t \int_F e^{-2\alpha t} p^b_{t-s}(x,y)^2 g(s,u(s,y))^2 \mu(dy)ds\right|^\frac{p}{2} \right]^\frac{1}{p},
\end{split}
\end{equation*}
where we have used \cite[Theorem B.1]{Khoshnevisan2014}. We treat the two integrals on the right-hand side separately. Firstly,
\begin{equation*}
\begin{split}
\Ebb &\left[\left|\int_0^t \int_F e^{-\alpha t}p^b_{t-s}(x,y) f(s,u(s,y)) \mu(dy)ds\right|^p\right]^\frac{1}{p}\\
&\leq \int_0^t \int_F e^{-\alpha t}p^b_{t-s}(x,y) \Ebb\left[\left| f(s,u(s,y))\right|^p\right]^\frac{1}{p} \mu(dy)ds\\
&\leq C\int_0^t \int_F e^{-\alpha t}p^b_{t-s}(x,y) \left( 1 + \Ebb\left[\left| u(s,y)\right|^p\right]^\frac{1}{p}\right) \mu(dy)ds\\
&\leq C\left( te^{-\alpha t} + \int_0^t \int_F e^{-\alpha t}p^b_{t-s}(x,y) \sup_{y' \in F}\left( \Ebb\left[\left| u(s,y')\right|^p\right]^\frac{1}{p} \right) \mu(dy)ds \right)\\
&\leq \frac{C}{e\alpha} + C\int_0^t e^{-\alpha (t-s)} \sup_{y \in F}\left( e^{-\alpha s} \Ebb\left[\left| u(s,y)\right|^p\right]^\frac{1}{p} \right)ds\\
&\leq \frac{C}{\alpha} \left( 1 + \sup_{(s,y) \in [0,t] \times F} \left( e^{-\alpha s} \Ebb\left[\left| u(s,y)\right|^p\right]^\frac{1}{p} \right) \right),
\end{split}
\end{equation*}
and secondly
\begin{equation*}
\begin{split}
\Ebb &\left[\left|\int_0^t \int_F e^{-2\alpha t} p^b_{t-s}(x,y)^2 g(s,u(s,y))^2 \mu(dy)ds\right|^\frac{p}{2} \right]^\frac{1}{p}\\
&\leq \left(\int_0^t \int_F e^{-2\alpha t} p^b_{t-s}(x,y)^2 \Ebb \left[\left| g(s,u(s,y))\right|^p \right]^\frac{2}{p} \mu(dy)ds \right)^\frac{1}{2}\\
&\leq C\left(\int_0^t \int_F e^{-2\alpha t} p^b_{t-s}(x,y)^2 \sup_{y' \in F}\left(1 + \Ebb \left[\left| u(s,y') \right|^p \right]^\frac{1}{p} \right)^2 \mu(dy)ds \right)^\frac{1}{2}\\
&\leq C\left(\int_0^t e^{-2\alpha (t-s)} p^b_{2(t-s)}(x,x) \sup_{y \in F}\left(e^{-\alpha s} + e^{-\alpha s}\Ebb \left[\left| u(s,y) \right|^p \right]^\frac{1}{p} \right)^2 ds \right)^\frac{1}{2}\\
&\leq C \sup_{(s,y) \in [0,t] \times F}\left(e^{-\alpha s} + e^{-\alpha s}\Ebb \left[\left| u(s,y) \right|^p \right]^\frac{1}{p} \right)\left(\int_0^t e^{-2\alpha (t-s)} p^b_{2(t-s)}(x,x) ds \right)^\frac{1}{2}\\
&\leq \frac{C}{\sqrt{2}} \sup_{(s,y) \in [0,t] \times F}\left(1 + e^{-\alpha s}\Ebb \left[\left| u(s,y) \right|^p \right]^\frac{1}{p} \right)\left(\int_0^\infty e^{-\alpha s} p^b_{s}(x,x) ds \right)^\frac{1}{2}\\
&\leq C \sqrt{\frac{c_8}{2\kappa(\alpha)}} \left( 1+ \sup_{(s,y) \in [0,t] \times F}\left( e^{-\alpha s}\Ebb \left[\left| u(s,y) \right|^p \right]^\frac{1}{p} \right) \right).\\
\end{split}
\end{equation*}
Putting these estimates together we see that for all $(t,x) \in [0,\infty) \times F$,
\begin{equation*}
\begin{split}
e^{-\alpha t}&\Ebb\left[ |u(t,x)|^p \right]^\frac{1}{p} \\
&\leq\sup_{x \in F}|u_0(x)| + C\left( \frac{1}{\alpha} + \sqrt{\frac{2c_8p}{\kappa(\alpha)}} \right) \left( 1+ \sup_{(s,y) \in [0,t] \times F}\left( e^{-\alpha s}\Ebb \left[\left| u(s,y) \right|^p \right]^\frac{1}{p} \right) \right).
\end{split}
\end{equation*}
Now we pick $\alpha$ such that $C\left( \frac{1}{\alpha} + \sqrt{\frac{2c_8p}{\kappa(\alpha)}} \right) \leq \frac{1}{2}$. By the asymptotic properties of $\kappa$, it is possible to find $c_{10} > 0$ such that the choice
\begin{equation*}
\alpha = \alpha(p) := c_{10} p^{(1-\frac{d_s}{2})^{-1}}
\end{equation*}
works for all $p \geq 2$, and note that $(1-\frac{d_s}{2})^{-1} = 1+d_H$. Thus for all $p \geq 2$,
\begin{equation*}
\sup_{(s,x) \in [0,t] \times F} \left( e^{-c_{10} p^{1+d_H} s}\Ebb \left[\left| u(s,x) \right|^p \right]^\frac{1}{p} \right) \leq 2\sup_{x \in F}|u_0(x)| + 1.
\end{equation*}
Let $c_9 = 2\sup_{x \in F}|u_0(x)| + 1$. Thus for all $p \geq 2$ and $(t,x) \in [0,\infty) \times F$,
\begin{equation*}
\Ebb \left[\left| u(t,x) \right|^p \right]^\frac{1}{p} \leq c_9 \exp\left(c_{10} p^{1+d_H} t\right).
\end{equation*}
The analogous result for $p \in [1,2)$ follows via Jensen's inequality (and a minor adjustment of the value of $c_{10}$).
\end{proof}

\subsection{Weak intermittency}

Consider the following SPDE on $F$ for time $t \in [0,\infty)$:
\begin{equation}\label{eqn:compSPDE2}
\begin{split}
\frac{\partial u}{\partial t}(t,x) &= \Delta_N u(t,x) + g(u(t,x))\xi(t,x),\\
u(0,x) &= u_0(x).
\end{split}
\end{equation}
\begin{hyp}\label{hyp:lowerhyp}
We make the following assumptions:
\begin{enumerate}
\item $u_0: F \to \Rbb$ is measurable, bounded and non-negative.
\item $g:\Rbb \to \Rbb$ is Lipschitz: there exists a constant $C > 0$ such that for all $x,y \in \Rbb$,
\begin{equation*}
\begin{split}
|g(x) - g(y)| &\leq C|x - y|,\\
|g(x)| &\leq C(1+|x|).
\end{split}
\end{equation*}
\end{enumerate}
\end{hyp}

\begin{thm}[Limiting second moment lower bound]\label{thm:2ndlower}
There exists $c_{11} > 0$ such that the following holds: Assume Hypothesis \ref{hyp:lowerhyp}. Let $u$ be the solution to \eqref{eqn:compSPDE2}, and let $I(t) = \inf_{x \in F}\Ebb\left[ u(t,x)^2\right]$ for $t \in [0,\infty)$. Then
\begin{equation*}
\liminf_{t \to \infty} \left[ \exp\left(-c_{11}\kappa^{-1}(L_g^2) t\right) I(t) \right] \geq \inf_{x \in F} u_0(x)^2,
\end{equation*}
where
\begin{equation*}
L_g = \inf_{z \in \Rbb \setminus \{ 0 \}} \left| \frac{g(z)}{z} \right|.
\end{equation*}
\end{thm}

\begin{proof}
This is similar to the proof of \cite[Theorem 7.8]{Khoshnevisan2014}. Notice that since $u_0$ is non-negative, for all $x \in F$ we have that 
\begin{equation*}
\int_F p^N_t(x,y) u_0(y) \mu(dy) \geq \inf_{y \in F} u_0(y) \int_F p^N_t(x,y) \mu(dy) = \inf_{y \in F} u_0(y).
\end{equation*}
Using the mild formulation of the solution we see that
\begin{equation*}
\begin{split}
&\Ebb\left[ u(t,x)^2 \right] \\
&= \left( \int_F p^N_t(x,y) u_0(y) \mu(dy) \right)^2 + \Ebb\left[\left( \int_0^t \int_F p^N_{t-s}(x,y) g(u(s,y)) \xi(s,y) \mu(dy)ds \right)^2\right]\\
&\geq \inf_{x \in F} u_0(x)^2 + \int_0^t \int_F p^N_{t-s}(x,y)^2 \Ebb\left[ g(u(s,y))^2\right] \mu(dy)ds\\
&\geq \inf_{x \in F} u_0(x)^2 + \int_0^t \int_F L_g^2 p^N_{t-s}(x,y)^2 \Ebb\left[ u(s,y)^2\right] \mu(dy)ds
\end{split}
\end{equation*}
for all $(t,x) \in [0,\infty) \times F$. It follows by Lemma \ref{lem:heatbounds} that
\begin{equation*}
\begin{split}
I(t) &\geq \inf_{x \in F} u_0(x)^2 + \int_0^t \int_F L_g^2 p^N_{t-s}(x,y)^2 I(s) \mu(dy)ds\\
&= \inf_{x \in F} u_0(x)^2 + \int_0^t L_g^2 p^N_{2(t-s)}(x,x) I(s) ds\\
&\geq \inf_{x \in F} u_0(x)^2 + \int_0^t 2^{-\frac{d_s}{2}}c_7L_g^2 \left( 1 + (t-s)^{-\frac{d_s}{2}} \right) I(s) ds.
\end{split}
\end{equation*}
Now if $L_g = 0$ then the result is clear. Henceforth we assume that $L_g > 0$. Recall that the function $\kappa$ is strictly increasing and bijective on $[0,\infty)$. Let
\begin{equation*}
\alpha = \kappa^{-1} \left(2^{-\frac{d_s}{2}}c_1L_g^2\right) > 0,
\end{equation*}
then we see that
\begin{equation*}
e^{-\alpha t}I(t) \geq e^{-\alpha t}\inf_{x \in F} u_0(x)^2 + \int_0^t \kappa(\alpha) e^{-\alpha(t-s)} \left( 1 + (t-s)^{-\frac{d_s}{2}} \right) e^{-\alpha s} I(s) ds
\end{equation*}
for all $t \in [0,\infty)$. Since
\begin{equation*}
\int_0^\infty \kappa(\alpha) e^{-\alpha s} \left( 1 + s^{-\frac{d_s}{2}} \right) ds = 1,
\end{equation*}
it follows that $t \mapsto e^{-\alpha t}I(t)$ is a non-negative supersolution (in the sense of \cite[Definition 7.10]{Khoshnevisan2014}) to the renewal-type equation
\begin{equation}\label{eqn:renewal}
f(t) = e^{-\alpha t}\inf_{x \in F} u_0(x)^2 + \int_0^t \kappa(\alpha) e^{-\alpha(t-s)} \left( 1 + (t-s)^{-\frac{d_s}{2}} \right) f(s) ds.
\end{equation}
By classical renewal theory \cite[Theorem XI.1.2]{Feller1966}, \eqref{eqn:renewal} has a unique non-negative bounded solution $f: [0,\infty) \to \Rbb$ for which
\begin{equation*}
\lim_{t \to \infty} f(t) = \frac{\int_0^\infty e^{-\alpha t}\inf_{x \in F} u_0(x)^2 dt}{\kappa(\alpha) \int_0^\infty t e^{-\alpha t} \left( 1 + t^{-\frac{d_s}{2}} \right) dt}.
\end{equation*}
We compute the upper integral and use an elementary result on derivatives of Laplace transforms \cite[XIII.2(ii)]{Feller1966} on the lower integral. This gives
\begin{equation*}
\begin{split}
\lim_{t \to \infty} f(t) &= \frac{\alpha^{-1}\inf_{x \in F} u_0(x)^2}{\kappa(\alpha) \cdot \frac{\kappa'(\alpha)}{\kappa(\alpha)^2}}\\
&= \frac{1 + \alpha^\frac{d_s}{2}\Gamma(1 - d_s / 2)}{1 + \alpha^\frac{d_s}{2}\Gamma(2 - d_s / 2)} \inf_{x \in F} u_0(x)^2\\
&= \frac{1 + \alpha^\frac{d_s}{2}\Gamma(1 - d_s / 2)}{1 + \alpha^\frac{d_s}{2}(1 - d_s / 2)\Gamma(1 - d_s / 2)} \inf_{x \in F} u_0(x)^2,
\end{split}
\end{equation*}
and it follows that
\begin{equation*}
\lim_{t \to \infty} f(t) \geq \inf_{x \in F} u_0(x)^2
\end{equation*}
for any $\alpha > 0$. Now observe: by Theorem \ref{thm:uppermmtbd} there exists $\beta \geq 0$ such that
\begin{equation*}
\sup_{t \in [0,\infty)}\left( e^{-\beta t} e^{-\alpha t}I(t) \right) < \infty.
\end{equation*}
Therefore by \cite[Lemma A.2]{Georgiou2015}, $e^{-\alpha t}I(t) \geq f(t)$ for all $t \in [0,\infty)$. Finally, there can easily be found a constant $c_{11} > 0$ such that
\begin{equation*}
2^{-\frac{d_s}{2}}c_7 \kappa(x) \geq \kappa(c_5x)
\end{equation*}
for all $x \geq 0$. Taking $x = \kappa^{-1}(L_g^2)$ this implies that $\alpha \geq c_{11}\kappa^{-1}(L_g^2)$. So
\begin{equation*}
\exp(-c_{11}\kappa^{-1}(L_g^2) t)I(t) \geq f(t)
\end{equation*}
for all $t \in [0,\infty)$. The result follows.
\end{proof}
\begin{cor}[Weak intermittency]
Assume Hypothesis \ref{hyp:lowerhyp}. Suppose that
\begin{equation*}
\begin{split}
\inf_{z \in \Rbb \setminus \{ 0 \}} \left| \frac{g(z)}{z} \right| &> 0,\\
\inf_{x \in F} u_0(x) &> 0.
\end{split}
\end{equation*}
Then the solution to \eqref{eqn:compSPDE2} is weakly intermittent in the sense of \cite[Definition 7.7]{Khoshnevisan2014}.
\end{cor}
\begin{proof}
Theorem \ref{thm:2ndlower} directly implies the required result.
\end{proof}

\bibliography{CRFbib2}
\bibliographystyle{alpha}

\end{document}